\documentclass[11pt]{article}
\pdfoutput=1
\usepackage{amstext,amsmath,amssymb,amsthm,constants} 
\usepackage{latexsym}
\usepackage{exscale}
\usepackage{color}
\usepackage{float}

\usepackage[margin=1in]{geometry} 
\usepackage{amsmath,amsthm,amssymb, graphicx, multicol, array}

\usepackage{pgfplots}
\pgfplotsset{compat=1.6}

\pgfplotsset{soldot/.style={color=blue,only marks,mark=*}} \pgfplotsset{holdot/.style={color=blue,fill=white,only marks,mark=*}}

\RequirePackage[cp1251]{inputenc}
\RequirePackage{srcltx}

\numberwithin{equation}{section}
\RequirePackage{srcltx}

\RequirePackage[cp1251]{inputenc}

\numberwithin{equation}{section}

\newconstantfamily{smc}{symbol=c}
\newconstantfamily{lgc}{symbol=C}

\usepackage{mathtools}

\newtheorem{Thm}{Theorem}[section]
\newtheorem{Lemma}[Thm]{Lemma}
\newtheorem{Cor}[Thm]{Corollary}

\theoremstyle{remark}
\newtheorem{Rem}{Remark}[section]

\renewcommand\hat{\widehat}
\def\XXint#1#2#3{{\setbox0=\hbox{$#1{#2#3}{\int}$ }
		\vcenter{\hbox{$#2#3$ }}\kern-.6\wd0}}

\title{Transitions for exceptional times in dynamical first-passage percolation}
\author{Michael Damron \thanks{The research of M. D. is supported by an NSF grant DMS-2054559 and an NSF CAREER award.} \\ \small{Georgia Tech}  \and Jack Hanson \thanks{The research of J. H. is supported by NSF grants DMS-1612921, DMS-1954257, a PSC-CUNY grant and a CUNY JFRASE award via the CUNY Office of Research and the Sloan Foundation.}\\ \small{City College of New York} \and David Harper \thanks{The research of D. H. is partially supported by NSF grant DMS-2054559.} \\ \small{Georgia Tech} \and Wai-Kit Lam \\ \small{National Taiwan University} \\ \small{University of Minnesota}}

\begin{document}
	
	\maketitle 
	\begin{abstract}
		In first-passage percolation (FPP), we let $(\tau_v)$ be i.i.d.~nonnegative weights on the vertices of a graph and study the weight of the minimal path between distant vertices. If $F$ is the distribution function of $\tau_v$, there are different regimes: if $F(0)$ is small, this weight typically grows like a linear function of the distance, and when $F(0)$ is large, the weight is typically of order one. In between these is the critical regime in which the weight can diverge, but does so sublinearly. We study a dynamical version of critical FPP on the triangular lattice where vertices resample their weights according to independent rate-one Poisson processes. We prove that if $\sum F^{-1}(1/2+1/2^k) = \infty$, then a.s.~there are exceptional times at which the weight grows atypically, but if $\sum k^{7/8} F^{-1}(1/2+1/2^k) <\infty$, then a.s.~there are no such times. Furthermore, in the former case, we compute the Hausdorff and Minkowski dimensions of the exceptional set and show that they can be but need not be equal. These results show a wider range of dynamical behavior than one sees in subcritical (usual) FPP.
	\end{abstract}

	\section{Introduction}	
	
	First-passage percolation (FPP) was introduced in the '60s by Hammersley-Welsh \cite{HW65} and is a prototypical example of a random growth model.  Such models give insight into numerous reaction-limited growth phenomena like the spread of tumors or bacterial colonies, and fluid flow in porous media \cite{HZ95}. In recent years, there has been much effort to understand the main properties of FPP, and this has led to connections with random matrices \cite{J00}, particle systems \cite{R81}, and the KPZ equation \cite{C18}. Nonetheless, there are still many fundamental open questions; see \cite{ADH17} for a recent survey.
	
	In 2015, Ahlberg introduced a dynamical version of FPP in which vertices resample their weights according to independent rate-one Poisson processes. His work is motivated by a number of important examples (including the dynamical web \cite{FNRS09,SSS17}, dynamical sensitivity of random sequence properties \cite{BHPS03}, and dynamical Bernoulli percolation \cite{HPS97}) where the behavior at random ``exceptional'' times can be dramatically different from the behavior at deterministic times. Ahlberg's result in \cite[Thm.~4]{A15} is that in the usual (subcritical) regime of FPP, there are no exceptional times when the passage time has an atypical linear growth rate. A natural question is whether the critical or supercritical regimes have exceptional times. In this paper, we focus on this question in critical FPP in two dimensions, where recent advances \cite{DHL19,DLW17,Y18} have allowed an exact characterization of the asymptotic growth rate in the static model. 
	
	Our results for dynamical critical FPP are stated in terms of the distribution function $F$ of the vertex weights. Writing $a_k = F^{-1}(1/2 + 1/2^k)$, we find that when $\sum a_k = \infty$, the model exhibits exceptional times (Theorem~\ref{thm: Hausdorff}), and when $\sum a_k <\infty$, under a weak decay condition on $a_k$, there are no exceptional times (Theorem~\ref{thm: other_side}). Furthermore, in the former case, we compute the Hausdorff dimension of the exceptional set and show that its Minkowski dimension has a transition (Theorem~\ref{thm: Minkowski}) depending on the behavior of the sequence $ka_k$. This implies that for some distributions, the exceptional set has the same Hausdorff and Minkowski dimensions, but for others, these dimensions are different.

	%

	\subsection{Critical FPP}
	
	
	We consider the model on the triangular lattice $\mathbb{T}$, embedded in $\mathbb{R}^2$ with vertex set $\mathbb{Z}^2$ and edges between points of the form $(v_1,w_1)$ and $(v_2,w_2)$ with either (a) $\|(v_1,w_1)-(v_2,w_2)\|_1 = 1$ or (b) both $v_2 = v_1 + 1$ and $w_2 = w_1 - 1$. We define i.i.d.~vertex weights $(\tau_v)_{v \in \mathbb{Z}^2}$ with some common distribution function $F$ satisfying $F(0^-) = 0$. It is typical to do this by letting $(\omega_v)_{v \in \mathbb{Z}^2}$ be a family of i.i.d.~uniform $[0,1]$ random variables, and setting $\tau_v = F^{-1}(\omega_v)$, where $F^{-1}$ is the generalized inverse
	\[
	F^{-1}(t) = \inf\{y \in \mathbb{R} : F(y) \geq t\} \text{ for } t \in (0,1).
	\]
	A path is a sequence of vertices $(v_1, \dots, v_n)$ with $v_i$ being adjacent to $v_{i+1}$ for all $i=1, \dots, n-1$. For a path $\gamma = (v_1, \dots, v_n)$, we define its passage time by
	\[
	T(\gamma) = \sum_{i=2}^n \tau_{v_i}
	\]
	(this definition naturally extends to infinite paths), and for vertex sets $A,B \subset \mathbb{Z}^2$, we define the first-passage time from $A$ to $B$ by
	\[
	T(A,B) = \inf\{T(\gamma) : \gamma \text{ is a path from a vertex in }A \text{ to a vertex in }B\}.
	\]
	If $A$ or $B$ is a singleton $\{v\}$, we replace it by $v$ in the notation: for instance, we write $T(v,w)$ for $T(\{v\},\{w\})$. Last, we put $B(n) = \{v \in \mathbb{Z}^2 : \|v\|_\infty \leq n\}$ and $\partial B(n) = \{v \in \mathbb{Z}^2 : \|v\|_\infty = n\}$.
	
	Of prime importance in FPP is the leading order growth rate of the passage time $T$ between distant vertices. An elementary argument using Fekete's lemma implies that for $x \in \mathbb{Z}^2$, the limit $g(x) = \lim_{n \to \infty} \mathbb{E}T(0,nx)/n$ exists under a mild moment condition on $\tau_v$. This can be improved to a.s.~and $L^1$ convergence using Kingman's subadditive ergodic theorem. Kesten proved \cite[Thm.~6.1]{K86} that $g$ is identically zero when $F(0) > 1/2$ (supercritical) or $F(0) = 1/2$ (critical), and that $g(x) > 0$ for all nonzero $x$ when $F(0) < 1/2$ (subcritical). In the supercritical regime, $(T(x,y))_{x,y \in \mathbb{Z}^2}$ forms a tight family, and in the subcritical regime, $T$ grows linearly. The limit $g$ does not give precise information about the growth of $T$ in the critical regime.
	
	The first step toward quantifying the growth in the critical case was made by Chayes-Chayes-Durrett \cite[Thm.~3.3]{CCD86}. They showed that when the $\tau_v$'s are Bernoulli$(1/2)$ random variables, the sequence $\mathbb{E}T(0,\partial B(n))$ is of order $\log n$. Kesten-Zhang \cite{KZ97} in '97 went further, showing that $\mathrm{Var}~T(0,\partial B(n))$ is of order $\log n$ and that $T(0,\partial B(n))$ satisfies a Gaussian central limit theorem. The results of \cite{KZ97} hold for a wider class of critical $F$: those satisfying $F(0) = F(\delta) = 1/2$ for some $\delta>0$. In these cases, Yao \cite{Y18} and Damron-Hanson-Lam \cite{DHL19} found the precise leading order behavior of $T(0,\partial B(n))$ and its variance: for instance 
	\[
	\frac{T(0,\partial B(n))}{\log n} \to \frac{I}{2\sqrt{3}\pi} \text{ a.s.},
	\]
	where $I$ is the infimum of all $\delta$ with $F(\delta)>1/2$. All of these works leave open the growth rate for $F$ satisfying $F(0) = 1/2$ but $I=0$.
	
	Zhang was the first to show in the case $I=0$ that the model displays ``double behavior.'' In \cite{Z99}, he exhibited examples of $F$ for which $F(0) = 1/2$ but the sequence $T(0,\partial B(n))$ remains bounded. Specifically, for $a>0$, if we define
	\[
	F_a(x) = \begin{cases}
		1 & \quad\text{if } x \geq \left( \frac{1}{2}\right)^{\frac{1}{a}} \\
		\frac{1}{2} + x^a &\quad\text{if } 0 \leq x \leq \left( \frac{1}{2}\right)^{\frac{1}{a}} \\
		0 & \quad\text{if } x < 0,
	\end{cases}
	\]
	then \cite[Thm.~8.1.1]{Z99} states that if $a$ is small enough, then the sequence $T(0,\partial B(n))$ is a.s.~bounded. Zhang also gave examples of $F$ with $I=0$ but such that $T(0,\partial B(n))$ diverges a.s.
	
	In '17, Damron-Lam-Wang \cite{DLW17} found an explicit condition that characterizes this double behavior. Assuming
	\begin{equation}\label{eq: critical_def}
		F(0) = \frac{1}{2}
	\end{equation}
	and $\mathbb{E}\tau_v^\alpha<\infty$ for some $\alpha > 1/6$, we have
	\begin{equation}\label{eq: asymptotics}
		\mathbb{E}T(0,\partial B(n)) \asymp \sum_{k=2}^{\lfloor \log_2 n \rfloor} F^{-1} \left( \frac{1}{2} + \frac{1}{2^k}\right),
	\end{equation}
	where $\asymp$ means that the ratio of the left and right sides is bounded away from zero and infinity as $n \to \infty$. Furthermore, under only \eqref{eq: critical_def}, if we define
	\begin{equation}\label{eq: rho_def}
		\rho = \lim_{n \to \infty} T(0,\partial B(n)),
	\end{equation}
	which can be written equivalently as
	\begin{equation}\label{eq: rho_equivalent}
		\rho = \inf\left\{ T(\Gamma) : \Gamma \text{ is an infinite self-avoiding path starting at }0\right\},
	\end{equation}
	then by the Kolmogorov zero-one law, $\mathbb{P}(\rho = \infty) \in \{0,1\}$ and
	\begin{equation}\label{eq: rho_equivalence}
		\rho < \infty~\text{a.s. } \Leftrightarrow \sum_{k=2}^\infty F^{-1} \left( \frac{1}{2} + \frac{1}{2^k} \right) < \infty.
	\end{equation}
	These results were stated on the square lattice but the proofs remain valid on $\mathbb{T}$.

	\subsection{Dynamical critical FPP}\label{sec: dynamical_intro}
	
	In this section, we introduce Ahlberg's version of dynamical FPP. Let $(\mathfrak{s}_v) = (\mathfrak{s}_v(t) : v \in \mathbb{Z}^2, t \in [0,\infty) )$ be a family of i.i.d.~rate one Poisson processes and let $(\omega_v^{(n)}: v \in \mathbb{Z}^2, n \geq 0)$ be a family of i.i.d.~uniform $[0,1]$ random variables. For a distribution function $F$ satisfying $F(0^-)=0$, set
	\[
	\tau_v(t) = F^{-1}(\omega_v^{(n)}) \text{ for } t \in [0,\infty) \text{ such that } \mathfrak{s}_v(t) = n.
	\]
	For fixed $t$, the family $(\tau_v(t))_{v \in \mathbb{Z}^2}$ is i.i.d.~with common distribution function $F$. We will also define $T_t(v,w)$ and $T_t(A,B)$ as the corresponding passage times in the environment $(\tau_v(t))_{v \in \mathbb{Z}^2}$. Last we set $\rho_t = \lim_{n \to \infty} T_t(0,\partial B(n))$.
	
	We will be interested in ``exceptional times,'' or times at which the dynamical model behaves differently than the static one.  Such questions were studied in \cite{A15}, where it was shown that in the subcritical regime, a.s.~there are no times at which $T$ displays atypical leading order growth. Precisely, for any $\epsilon>0$, under a mild moment condition on $\tau_v$, one has
	\[
	\sum_{z \in \mathbb{Z}^2} \mathbb{P}\left( \sup_{t \in [0,1]} \left| T_t(0,z) - g(z)\right| > \epsilon |z|\right) < \infty.
	\]
	This result extends to dimensions bigger than two. Ahlberg's work on the subcritical regime motivates the question of whether exceptional times exist for critical or supercritical FPP. The present paper initiates the study of the dynamical behavior of critical FPP.
	
	Due to the dichotomy in \eqref{eq: rho_equivalence}, exceptional has a different meaning depending on which of the conditions
	\begin{equation}\label{eq: finite_sum}
		\sum_{k=2}^\infty F^{-1} \left( \frac{1}{2} + \frac{1}{2^k} \right) < \infty
	\end{equation}
	and
	\begin{equation}\label{eq: infinite_sum}
		\sum_{k=2}^\infty F^{-1} \left( \frac{1}{2} + \frac{1}{2^k} \right) = \infty
	\end{equation}
	holds. Namely, under \eqref{eq: critical_def}, we define the exceptional sets
	\begin{equation}\label{eq: exceptional_def}
		\begin{split} 
			& \{t \geq 0 : \rho_t = \infty\} \text{ under \eqref{eq: finite_sum}} \text{ and} \\
			& \{t \geq 0 : \rho_t \leq x\} \text{ under \eqref{eq: infinite_sum} for } x \in [0,\infty).
		\end{split}
	\end{equation}

	%
	%
	%
	%
	%
	%
	%
	%
	%
	Note that by time-stationarity and \eqref{eq: rho_equivalence}, any fixed $t \geq 0$ a.s.~is not in any of these exceptional sets: under \eqref{eq: critical_def}, for any $t \geq 0$, $\mathbb{P}(\rho_t \leq x) = 0$ for $x \in [0,\infty)$ when \eqref{eq: infinite_sum} holds, and $\mathbb{P}(\rho_t = \infty) = 0$ when \eqref{eq: finite_sum} holds. We can therefore apply Fubini's theorem to find that their Lebesgue measures are zero a.s. Thus we are led to consider their fractal dimensions.
	
	We recall the different notions of dimension that we will consider; we quote their definitions from \cite[Sec.~14.1]{LP16}. If $E$ is a bounded subset of $\mathbb{R}^d$ for some $d \geq 1$ and $\epsilon>0$, let $N(E,\epsilon)$ be the minimum number of balls of diameter at most $\epsilon$ required to cover $E$. Then the upper and lower Minkowski dimensions of $E$ are
	\begin{eqnarray*}
		\text{dim}^\text{M}(E) &=\limsup_{\epsilon \to 0} \frac{\log N(E,\epsilon)}{\log \frac{1}{\epsilon}} \\
		\text{dim}_\text{M}(E) &=\liminf_{\epsilon \to 0} \frac{\log N(E,\epsilon)}{\log \frac{1}{\epsilon}}.
	\end{eqnarray*}
	For $\alpha>0$ and $E \subset \mathbb{R}^d$ that is possibly unbounded, the Hausdorff $\alpha$-dimensional (outer) measure of $E$ is
	\begin{equation}\label{eq: hausdorff_definition}
		\mathcal{H}_\alpha(E) = \lim_{\epsilon \to 0} ~\inf\left\{ \sum_{i=1}^\infty \left( \text{diam }E_i\right)^\alpha : E \subset \bigcup_{i=1}^\infty E_i \text{ and diam }E_i < \epsilon \text{ for all }i\right\}.
	\end{equation}
	For any $E \subset \mathbb{R}^d$, there is a number $\alpha_0$ such that if $\alpha < \alpha_0$ then $\mathcal{H}_\alpha(E) = \infty$ and if $\alpha > \alpha_0$ then $\mathcal{H}_\alpha(E) = 0$. This number is called the Hausdorff dimension of $E$, and we write it as $\text{dim}_\text{H}(E)$. It is standard that for bounded $E$,
	\begin{equation}\label{eq: dimension_inequalities}
		\text{dim}^\text{M}(E) \geq \text{dim}_\text{M}(E) \geq \text{dim}_\text{H}(E).
	\end{equation}
	Furthermore, Hausdorff dimension has the ``countable stability'' property, which states that if $E_1, E_2, \dots$ are subsets of $\mathbb{R}^d$, then
	\[
	\text{dim}_\text{H}\left( \bigcup_{i=1}^\infty E_i\right) = \sup\left\{ \text{dim}_\text{H}(E_i) : i \geq 1\right\}.
	\]

	\subsection{Main results}\label{sec: main_results}
	To state our results, we use the shorthand
	\[
	a_k = F^{-1}\left( \frac{1}{2} + \frac{1}{2^k}\right) \text{ for } k \geq 2,
	\]
	so that \eqref{eq: finite_sum} and \eqref{eq: infinite_sum} become $\sum a_k < \infty$ and $\sum a_k = \infty$ respectively. The first theorem gives the Hausdorff dimension of the set of exceptional times when $a_k$ is not summable.
	\begin{Thm}\label{thm: Hausdorff}
		If $F$ satisfies \eqref{eq: critical_def} and \eqref{eq: infinite_sum}, then
		\[
		\textnormal{dim}_\textnormal{H}\left( \{t \geq 0 : \rho_t < \infty\}\right) = \frac{31}{36} \text{ a.s.}
		\]
	\end{Thm}
	We will prove Theorem~\ref{thm: Hausdorff} in Section~\ref{sec: Hausdorff}. Alternatively, one can consider the subset $\mathcal{E}_x = \{t \geq 0 : \rho_t \leq x\}$ for $x \geq 0$. In the case $x=0$, this is the same as the set of exceptional times for critical Bernoulli percolation, which has Hausdorff dimension $31/36$ a.s. \cite{GPS10,SS10}. By monotonicity and Theorem~\ref{thm: Hausdorff}, we also obtain $\text{dim}_\textnormal{H}~\mathcal{E}_x = 31/36$ for all $x \in [0,\infty)$ a.s.

	
	
	The next theorem shows that the upper Minkowski dimension of the set of exceptional times can differ from the Hausdorff dimension if $a_k$ is not summable, but decays quickly enough to 0. Because Minkowski dimension is defined for bounded sets (although there are some modifications for unbounded sets), we intersect the set of exceptional times with $[0,s]$ and take $s$ to infinity. 
	
	\begin{Thm}\label{thm: Minkowski}
		Suppose that $F$ satisfies \eqref{eq: critical_def} and \eqref{eq: infinite_sum}.
		\begin{enumerate}
			\item If $ka_k \to \infty$, then for any $x \in [0,\infty)$,
			\[
			\mathbb{P}\left( \textnormal{dim}^\textnormal{M}\left( \{t \in [0,s] : \rho_t \leq x\}\right) = \frac{31}{36} \right) \to 1 \text{ as } s \to \infty.
			\]
			\item If $\liminf_{k \to \infty} ka_k = 0$,
			then for any $x \in (0,\infty)$,
			\[
			\mathbb{P}\left( \textnormal{dim}^\textnormal{M}\left( \{t \in [0,s] : \rho_t \leq x\}\right) = 1 \right) \to 1 \text{ as } s \to \infty.
			\]
		\end{enumerate}
	\end{Thm}
	We will prove Theorem~\ref{thm: Minkowski} in Section~\ref{sec: Minkowski}. We note that for any $F$ satisfying \eqref{eq: critical_def}, the set $\cup_{x \in \mathbb{N}} \{t \in [0,s] : \rho_t \leq x\} = \{t \in [0,s] : \rho_t <\infty\}$ contains the set of $t \in [0,s]$ for which there exists an infinite component of vertices of weight 0. By results of \cite{GPS10,SS10}, the latter is a.s.~dense in $[0,s]$ for $s>0$ and therefore has Minkowski dimension 1. Because Minkowski dimension is not countably stable, this fact does not contradict Theorem~\ref{thm: Minkowski}. 
	
	The proofs of Theorems~\ref{thm: Hausdorff} and \ref{thm: Minkowski} can also be used in the intermediate regime where $ka_k$ does not converge to 0 or $\infty$. From them, we obtain that for some $F$ satisfying \eqref{eq: critical_def} and \eqref{eq: infinite_sum}, 
	\[
	(\dagger)
	\begin{array}{c}
		\text{ the set of exceptional times has different upper and lower Minkowski dimensions, and} \\
		\text{ the upper Minkowski dimension of the set of $t$ where $\rho_t \leq x$ depends on $x$.}
	\end{array}
	\]
	We give formal statements of these facts in Section~\ref{sec: formal_statements} and briefly indicate how to obtain them from the above proofs.

	%
	%

	Our last result covers the case when $\sum a_k <\infty$. Here, we find the set of exceptional times is empty if $a_k$ does not decay very slowly.
	\begin{Thm}\label{thm: other_side}
		Suppose that $F$ satisfies \eqref{eq: critical_def} and \eqref{eq: finite_sum}. Then
		\[
		\sum_{k=2}^\infty k^{\frac{7}{8}} a_k < \infty \Rightarrow \text{a.s., } \{t \geq 0 : \rho_t = \infty\} = \emptyset.
		\]
	\end{Thm}
	
	\noindent
	We will prove Theorem~\ref{thm: other_side} in Section~\ref{sec: other_side}.
	
	The proof of Theorem~\ref{thm: other_side} gives a slightly stronger result than what is listed above: there exists $\varsigma>0$ such that if $\sum k^{7/8-\varsigma}a_k < \infty$, then a.s.~there are no exceptional times. One simply needs to modify the second to last line of the proof (below \eqref{eq: seriously_last}) in a straightforward manner. 
	Also, if one uses the conjectured value $17/48$ for the monochromatic two-arm exponent $\alpha_2'$ listed above \eqref{eq: max_exponent_bound}, then $\alpha = \min\{\alpha_2',1/3\}$ becomes $1/3$. Inserting this value into the argument below \eqref{eq: seriously_last} produces the following result. If $\sum k^{5/6 + \varsigma} a_k < \infty$ for some $\varsigma>0$, then a.s.~there are no exceptional times.

	Most of our arguments go through for general lattices (and half planes or sectors with the triangular lattice), but at some key points, we use properties of percolation on the triangular lattice derivable from the description of its scaling limit. Versions of our main results hold for, say, first-passage percolation on the edges of the square lattice, but they are weaker. For example, the exponent $7/8$ in Theorem~\ref{thm: other_side} comes from the exact values $1/4$ and $1/3$ of the two-arm polychromatic exponent in the full plane and the one-arm exponent in the half plane. These produce the bound $\alpha > 1/4$. On the square lattice, where these values are not rigorously established, we can only use $\alpha>0$, and this results in the condition $\sum ka_k<\infty$ instead of $\sum k^{7/8} a_k < \infty$.

	\section{Preliminaries}
	
	In this section we record various definitions and preliminary results that we will reference later. A circuit is a path $(v_1, \dots, v_n)$ with $v_1 = v_n$. Any circuit defines a plane curve by connecting its vertices in sequence with line segments. If the circuit is vertex self-avoiding, then this curve is a Jordan curve and therefore splits $\mathbb{R}^2$ into two components. The bounded component is referred to as the circuit's interior. 
	
	Given $p \in [0,1]$, we say that a vertex $w$ is $p$-open in the configuration $(\omega_v)_{v \in \mathbb{Z}^2}$ if $\omega_w \leq p$; otherwise we say it is $p$-closed. Because the variables $\omega_v$ are i.i.d.~and have uniform $[0,1]$ distribution, the collection of $p$-open vertices has the same distribution as the collection of open vertices in Bernoulli site-percolation with parameter $p$. We say that a path or circuit is $p$-open if all its vertices are $p$-open, and $p$-closed if all its vertices are $p$-closed. A circuit $\mathcal{C}$ in an annulus $\text{Ann}(m,n) = B(n) \setminus B(m)$ for integers $n \geq m \geq 1$ is said to be the innermost $p$-open circuit around 0 in this annulus if $\mathcal{C}$ is $p$-open, the origin is in the interior of $\mathcal{C}$, and any other $p$-open circuit in this annulus containing the origin in its interior has interior which contains the interior of $\mathcal{C}$. It is a well-known and often-used fact that if there is a $p$-open circuit around 0 in $\text{Ann}(m,n)$, then there is a unique innermost one (similar to \cite[p.~317]{G99}).
	
	We will heavily use tools from critical and near-critical percolation. To describe these we consider the probability of box crossings. A path in $B(n)$ is a left-right crossing if it is contained in $B(n)$, begins on the left side $\{-n\} \times [-n,n]$, and ends on the right side $\{n\} \times [-n,n]$. It is known that $1/2$ is a critical value in the sense that $\mathbb{P}(B(n)\text{ contains a left-right }p\text{-open crossing})$ converges to 0 if $p < 1/2$ and to $1$ if $p > 1/2$, but remains bounded away from 0 and 1 if $p = 1/2$. The finite-size scaling correlation length is defined by letting $\epsilon>0$ and setting
	\[
	L(p,\epsilon) = \min\{ n : \mathbb{P}(B(n)\text{ contains a left-right }p\text{-open crossing}) > 1-\epsilon\} \text{ for } p > \frac{1}{2}
	\]
	and
	\[
	L(p,\epsilon) = \min\{n : \mathbb{P}(B(n)\text{ contains a left-right }p\text{-open crossing}) < \epsilon\} \text{ for } p < \frac{1}{2}.
	\]
	For a given $\epsilon$, one has $L(p,\epsilon) \to \infty$ as $p \to 1/2$ and it is proved \cite[Eq.~(1.24)]{K87} that for some $\epsilon_0$, one has $L(p,\epsilon_1) \asymp L(p,\epsilon_2)$ for any fixed $\epsilon_1,\epsilon_2 \in [0,\epsilon_0]$, as $p\to 1/2$. We then define 
	\[
	L(p) = L(p,\epsilon_0).
	\]
	As $p\to 1/2$, $L(p)\to\infty$, and in fact
	\begin{equation}\label{eq: correlation_scaling}
		L(p) = \bigg| p-\frac{1}{2} \bigg|^{-\frac{4}{3} + o(1)} \text{ as } p \to \frac{1}{2}.
	\end{equation}
	This follows from the scaling relation \cite[Eq.~(4.5)]{K87} and the exact value $5/4$ of the four-arm exponent (see \eqref{eq: arm_exponents_1_2_4}). With this definition of $L(p)$, one can use the Russo-Seymour-Welsh theorem to show that for positive integers $k,l,n$, and $p$ such that $L(p) \geq n$, one has
	\begin{equation}\label{eq: near_critical_crossing}
		1-\delta_{k,l} > \mathbb{P}([0,kn]\times [0,ln] \text{ contains a left-right }p\text{-open crossing}) > \delta_{k,l}
	\end{equation}
	for some positive constants $\delta_{k,l}$ depending only on $k$ and $l$.
	
	We also set
	\[
	p_n = \sup\{p : L(p) > n\} \text{ for } n \geq 1.
	\]
	Although there are two values of $p_n$, one above $1/2$, and one below $1/2$, the one we are using will be made clear in applications, and the ensuing remarks are valid for both. This $p_n$ is nearly an inverse for $L(p)$: as in \cite[Eq.~(2.10)]{J03}, there is $\Cl[smc]{c: c1} \in (0,1)$ such that for all $n \geq 1$,
	\begin{equation}\label{eq: approx_inverse}
		\Cr{c: c1}n \leq L(p_n) \leq n.
	\end{equation}
	Furthermore, from \eqref{eq: correlation_scaling}, one has
	\begin{equation}\label{eq: p_n_bounds}
		\left| p_n - \frac{1}{2} \right| = n^{-\frac{3}{4} + o(1)} \text{ as } n \to \infty.
	\end{equation}
	
	Many of our arguments involve arm events, which are percolation events defined by the existence of paths (``arms'') emanating from a fixed region of space. The one-arm probability is defined as 
	\[
	\pi_1(p;m,n) = \mathbb{P}\left( \exists~p\text{-open path crossing Ann}(m,n)\right)
	\]
	and we use the shorthand
	\[
	\pi_1(p;n) = \pi_1(p;0,n),\quad \pi_1(n) = \pi_1\left(\frac{1}{2};n\right),\quad \text{ and } \pi_1(m,n) = \pi_1\left(\frac{1}{2};m,n\right).
	\]
	We make similar definitions for multi-arm events. $\pi_2(n)$ is the probability that there are two disjoint paths connecting $0$ to $\partial B(n)$, one which is $1/2$-open, and one which is $1/2$-closed, and $\pi_4(m,n)$ is the probability that there are four disjoint paths crossing $\text{Ann}(m,n)$, two of which are $1/2$-open, and two of which are $1/2$-closed. Furthermore, the paths alternate (open, closed, open, closed), as we proceed around the annulus in a clockwise fashion. The critical exponents for these events are known: for
	\begin{equation}\label{eq: arm_exponents_1_2_4}
		\alpha_1 = \frac{5}{48},~\alpha_2 = \frac{1}{4},~\alpha_4 = \frac{5}{4},
	\end{equation}
	one has \cite{SW01}
	\[
	\pi_k(n) = n^{-\alpha_k + o(1)} \text{ as } n \to \infty,
	\]
	and, more generally \cite[Lem.~2.5]{BN21}, for any $\epsilon>0$, there exist $\Cl[smc]{c: c2},\Cl[lgc]{C: C1}>0$ such that for all $m,n$ with $n > m \geq 1$,
	\begin{equation}\label{eq: annulus_arm_probability}
		\Cr{c: c2}\left( \frac{m}{n}\right)^{\alpha_k+\epsilon} \leq \pi_k(m,n) \leq \Cr{C: C1}\left( \frac{m}{n}\right)^{\alpha_k - \epsilon}.
	\end{equation}
	Arm probabilities also satisfy a ``quasimultiplicativity'' property: for $k=1,2,4$, there exist $\Cl[smc]{c: c_quasi},\Cl[lgc]{C: C_quasi}$ such that for any $m,n,r$ satisfying $n > r > m \geq 0$,
	\begin{equation}\label{eq: quasimultiplicativity}
		\Cr{c: c_quasi} \pi_k(m,r) \pi_k(r,n)\leq\pi_k(m,n) \leq \Cr{C: C_quasi}\pi_k(m,r)\pi_k(r,n).
	\end{equation}
	Furthermore, arm event probabilities remain nearly constant as $p$ changes, so long as it is near critical. Specifically, from \cite[Thm.~27]{N08}, for $k=1,2,4$, there exist $\Cl[smc]{c: change_p},\Cl[lgc]{C: change_p}$ such that for any $m,n$ with $n > m \geq 0$ and any $p$ with $L(p) \geq n$,
	\begin{equation}\label{eq: near_critical_arm_event_comparability}
		\Cr{c: change_p}\pi_k(m,n) \leq \pi_k(p;m,n) \leq \Cr{C: change_p}\pi_k(m,n).
	\end{equation}
	
	Measurability of all exceptional sets and their dimensions follows from routine arguments; see \cite[p.~504]{HPS97} for an analogous discussion in the context of dynamical percolation.
	
	We close this section with an elementary lemma that will be used in the proofs of Theorems~\ref{thm: Hausdorff} and~\ref{thm: other_side}.
	\begin{Lemma}\label{lem: sum_by_parts}
		Suppose that $(\mathsf{a}_n),(\mathsf{b}_n)$, and $(\mathsf{c}_n)$ are nonnegative sequences such that $(\mathsf{b}_n)$ is nonincreasing. Writing $\mathsf{A}_n = \sum_{k=1}^n \mathsf{a}_k$ for $n \geq 1$ (and similarly for $\mathsf{C}_n$), suppose that $\mathsf{A}_n \leq \mathsf{C}_n$ for all $n$. Then
		\[
		\sum_{k=1}^n \mathsf{a}_k\mathsf{b}_k \leq \sum_{k=1}^n\mathsf{c}_k\mathsf{b}_k \text{ for all } n \geq 1.
		\]
	\end{Lemma}
	\begin{proof}
		Because $\mathsf{A}_k(\mathsf{b}_k-\mathsf{b}_{k+1}) \leq \mathsf{C}_k(\mathsf{b}_k-\mathsf{b}_{k+1})$, we can apply summation by parts for
		\[
		\sum_{k=1}^n \mathsf{a}_k\mathsf{b}_k = \mathsf{A}_n\mathsf{b}_n + \sum_{k=1}^{n-1} \mathsf{A}_k(\mathsf{b}_k-\mathsf{b}_{k+1}) \leq \mathsf{C}_n\mathsf{b}_n + \sum_{k=1}^{n-1}\mathsf{C}_k(\mathsf{b}_k-\mathsf{b}_{k+1}) = \sum_{k=1}^n \mathsf{c}_k\mathsf{b}_k.
		\]
	\end{proof}

	\section{Proofs}
	
	\subsection{Proof of Theorem~\ref{thm: other_side}}\label{sec: other_side}
	
	In this section, we assume \eqref{eq: critical_def} and \eqref{eq: finite_sum}. By time-stationarity, we can show that a.s., there are no times $t \in [0,1]$ for which $\rho_t < \infty$. To do this, we define a sequence $(\mathsf{T}(n))$ of random variables with the property that
	\begin{equation}\label{eq: t_max_implication}
		\sum_{n=1}^\infty \mathsf{T}(n) < \infty \Rightarrow \rho < \infty.
	\end{equation}
	The tails of the distributions of the $\mathsf{T}(n)$'s are well-controlled, so we can give sufficient conditions for their sum to be finite for all times $t$. 
	
	For $n \geq 1$, define
	\[
	\mathsf{T}^{(1)}(n) = \min\left\{ \mathsf{T}(\Gamma) : \Gamma \text{ is a circuit around }0 \text{ in } B(2^{n+1}) \setminus B(2^n)\right\}
	\]
	and
	\[
	\mathsf{T}^{(2)}(n) = \min\left\{ \mathsf{T}(\Gamma) : \Gamma \text{ is a path that connects }B(2^n) \text{ to } \partial B(2^{n+2}) \text{ in } B(2^{n+2}) \setminus B(2^n)\right\}.
	\]
	Next, put
	\begin{equation}\label{eq: T(n)_def}
		\mathsf{T}(n) = \mathsf{T}^{(1)}(n) + \mathsf{T}^{(2)}(n).
	\end{equation}
	As usual, we add a subscript $t$ to these variables when they are evaluated in the configuration $(\tau_v(t))$. To see that \eqref{eq: t_max_implication} holds, suppose that $\sum_{n=1}^\infty \mathsf{T}(n) < \infty$. Then for each $n$, choose $\Gamma^{(1)}(n)$ and $\Gamma^{(2)}(n)$ to be minimizing for the definitions of $\mathsf{T}^{(1)}(n)$ and $\mathsf{T}^{(2)}(n)$. By planarity, the union $\cup_{n \geq 1, i=1,2} \Gamma^{(i)}(n)$ is an infinite, connected set of vertices with finite total passage time. We can then choose an infinite self-avoiding path starting at 0 that is contained in this union except for finitely many vertices. This path has finite passage time and so by \eqref{eq: rho_equivalent}, this shows \eqref{eq: t_max_implication}.
	
	The rest of the proof will be split into two sections. In the first, we give a tail bound for the $\mathsf{T}(n)$'s, and in the second, we use this bound to argue that if $\sum k^{7/8} a_k < \infty$, then a.s., for all $t \in [0,1]$, one has $\sum_n \mathsf{T}_t(n)<\infty$, and therefore by \eqref{eq: t_max_implication}, there are no exceptional times.

	\subsubsection{Tail bound for $\mathsf{T}(n)$}
	
	To estimate the probability that $\mathsf{T}(n)$ is large, we need to use more arm exponents than just those associated to 1, 2, and 4-arm events. For $m \leq n$,
	\begin{enumerate}
		\item $\pi_1^H(m,n)$ (one-arm half-plane probability) is the probability that there is a $1/2$-open path connecting $B(m)$ to $\partial B(n)$, but with all its vertices $v$ satisfying $v \cdot e_2 \geq 0$, and
		\item $\rho_2(m,n)$ (monochromatic two-arm probability) is the probability that there are two disjoint $1/2$-open paths connecting $B(m)$ to $\partial B(n)$ (without any second-coordinate restriction).
	\end{enumerate}
	Regarding these arm probabilities, we need the following facts:
	\begin{itemize}
		\item The half-plane one-arm exponent is $1/3$; see \cite[Thm.~22]{N08}.
		\item The monochromatic two-arm exponent $\alpha_2'$ satisfies $\alpha_2' \in (1/4, 2/3)$ \cite[Thms.~2 and 5]{BN11}. In fact, numerical evidence suggests $\alpha_2'=17/48$, but there is no rigorous proof --- see \cite{BN09}.
		\item Let 
		$\alpha = \min\{\alpha_2',1/3\}$. Similarly to \eqref{eq: annulus_arm_probability},
		for any $\epsilon>0$, there exists $\Cl[lgc]{C: max_arm}>0$ such that for all $m_1 \leq m_2$, 
		\begin{equation}\label{eq: max_exponent_bound}
			\max\{\pi_1^H(m_1,m_2), \rho_2(m_1,m_2)\} \leq \Cr{C: max_arm} \left(\frac{m_1}{m_2} \right)^{\alpha-\epsilon}.
		\end{equation}
	\end{itemize}
	
	%
	%
	
	Our tail inequality for $\mathsf{T}(n)$ will follow from a similar tail inequality for a rectangle passage time, after applying a straightforward gluing argument. Let $T(n)$ be the minimal passage time among all paths which remain in $R(n) = [-2^{n+1},2^{n+1}] \times [-2^n,2^n]$ and connect the left side of $R(n)$ to the right side of $R(n)$.
	\begin{Thm}\label{thm: upper_tail_theorem}
		Let $\epsilon \in (0,\alpha)$. There exist $\Cl[smc]{c: main_tail},\Cl[lgc]{C: main_tail}> 0$ such that the following holds. For all $n \geq 1$ and $p > 1/2$ with $L(p) \leq 2^n$,
		\begin{align*}
			\mathbb{P}\left( T(n) \geq \lambda F^{-1}(p) \left( \frac{2^n}{L(p)}\right)^{2-\alpha+\epsilon}\right) &\leq \exp\left( - \Cr{c: main_tail} \frac{2^n}{L(p)}\right) \\
			&+ \begin{cases}
				\exp\left( -\Cr{c: main_tail} \lambda^{\frac{2}{\alpha-\epsilon}}\right) &\quad \text{if } \lambda \leq \Cr{C: main_tail} \left( \frac{2^n}{L(p)}\right)^{\alpha-\epsilon} \\
				\exp\left( - \Cr{c: main_tail} \left( \frac{2^n}{L(p)}\right)^{2-\alpha+\epsilon} \lambda\right) &\quad \text{if } \lambda \geq \Cr{C: main_tail} \left( \frac{2^n}{L(p)}\right)^{\alpha-\epsilon}.
			\end{cases}
		\end{align*}
	\end{Thm}
	\begin{Rem}
		When $\lambda \geq \Cr{C: main_tail}\left(2^n/L(p)\right)^{\alpha-\epsilon}$, the term $\exp(-\Cr{c: main_tail} (2^n/L(p))^{2-\alpha+\epsilon}\lambda)$ is less than $\exp(-\Cr{c: main_tail}(2^n/L(p)))$. We write the above form of the inequality because the second summand is actually an upper bound for the probability $\mathbb{P}\left(\#V_n(p) \geq \lambda (2^n/L(p))^{2-\alpha+\epsilon}\right)$; see \eqref{eq: to_show_N_n}.
	\end{Rem}
	\begin{proof}
		For the proof, we follow \cite[Prop.~1]{DT19}, but with improvements using ideas from the combinatorial argument of Kiss \cite{K14}. 
		Let $S(n) = [-2^{n+2},2^{n+2}] \times [-2^n,2^n]$ and for $p>1/2$ and $n \geq 1$, define $E_n(p)$ as the event that there exists a $p$-open path in $S(n)$ connecting the left side of $S(n)$ to the right side. By the RSW theorem and \cite[Eq.~(2.8)]{J03}, we have, for some $\Cl[smc]{c: Jarai}>0$,
		\begin{equation}\label{eq: jarai_closed_bound}
			\mathbb{P}(E_n(p)^c) \leq \exp\left( - \Cr{c: Jarai} \frac{2^n}{L(p)}\right) \text{ for } n \geq 1, p>\frac{1}{2}.
		\end{equation}
		
		On the event $E_n(p)$, we let $T_p(n)$ be the minimal passage time among all paths in $S(n)$ that are $p$-open and connect the left side of $S(n)$ to the right side. Because we are interested in crossings of $R(n)$ only, we set
		\[
		\hat{T}_p(n) = \max\left\{ 
		\sum_{v \in \gamma \cap R(n)} \tau_v : 
		\begin{array}{c}
			\gamma \text{ is a }p\text{-open path in }S(n), \text{ it connects the left side}  \\
			\text{of }S(n) \text{ to the right side, and }T(\gamma) = T_p(n)
		\end{array}\right\}.
		\]
		By definition,
		\begin{equation}\label{eq: trivial_T_bound}
			T(n)\mathbf{1}_{E_n(p)} \leq \hat{T}_p(n)\mathbf{1}_{E_n(p)}.
		\end{equation}
		
		Vertices whose weights contribute to $\hat{T}_p(n)$ must satisfy certain conditions like those in arm events. For $v \in R(n)$, let $A_n(p,v)$ be the event that all of the following occur:
		\begin{enumerate}
			\item $\omega_v \in (1/2,p]$,
			\item there exist two disjoint $p$-open paths in $S(n)$ from $v$ to $\partial B(v,L(p))$, the boundary of the box $v + [-L(p),L(p)]^2$, and
			\item there exist two disjoint $1/2$-closed paths from $v$ to $\partial S(n)$, one touching the top side of $S(n)$ and one touching the bottom.
		\end{enumerate}
		Also, let
		\[
		V_n(p) = \{ v \in R(n) : A_n(p,v) \text{ occurs}\}.
		\]
		Exactly the same argument as in \cite[Lem.~1]{DT19} gives that for all $n \geq 1$ and $p > 1/2$ with $L(p) \leq 2^n$,
		\[
		\hat{T}_p(n)\mathbf{1}_{E_n(p)} \leq F^{-1}(p) \#V_n(p) \mathbf{1}_{E_n(p)}.
		\]
		By this inequality, \eqref{eq: trivial_T_bound}, and \eqref{eq: jarai_closed_bound}, we obtain for $\lambda \geq 0$, and $n,p$ as above,
		\begin{align*}
			\mathbb{P}\left( T(n) \geq \lambda F^{-1}(p) \left( \frac{2^n}{L(p)}\right)^{2-\alpha+\epsilon}\right) &\leq \mathbb{P}\left( T(n)\mathbf{1}_{E_n(p)} \geq  \lambda F^{-1}(p) \left( \frac{2^n}{L(p)}\right)^{2-\alpha+\epsilon}\mathbf{1}_{E_n(p)}\right) \\
			&+ \mathbb{P}(E_n(p)^c) \\
			&\leq \mathbb{P}\left( \#V_n(p) \geq \lambda \left( \frac{2^n}{L(p)}\right)^{2-\alpha+\epsilon} \right) \\
			&+ \exp\left( - \Cr{c: Jarai} \frac{2^n}{L(p)}\right).
		\end{align*}
		Comparing this with the statement of the theorem, we are therefore reduced to showing that for $n \geq 1$ and $p>1/2$ with $L(p) \leq 2^n$,
		\begin{equation}\label{eq: to_show_N_n}
			\mathbb{P}\left( \#V_n(p) \geq \lambda \left( \frac{2^n}{L(p)}\right)^{2-\alpha+\epsilon} \right) \leq \begin{cases}
				\exp\left( -\Cr{c: main_tail} \lambda^{\frac{2}{\alpha-\epsilon}}\right) &\quad \text{if } \lambda \leq \Cr{C: main_tail} \left( \frac{2^n}{L(p)}\right)^{\alpha-\epsilon} \\
				\exp\left( - \Cr{c: main_tail} \left( \frac{2^n}{L(p)}\right)^{2-\alpha+\epsilon} \lambda\right) &\quad \text{if } \lambda \geq \Cr{C: main_tail} \left( \frac{2^n}{L(p)}\right)^{\alpha-\epsilon}.
			\end{cases}
		\end{equation}
		
		To prove \eqref{eq: to_show_N_n}, we will estimate the moments of $\#V_n(p)$. The argument of \cite[Lem.~3]{DT19} gives the upper bound $\mathbb{E}\#V_n(p)^k\mathbf{1}_{\{\#V_n(p) \geq k\}} \leq \left( \Cl[lgc]{C: new_main_tail} k (2^n/L(p))^2\right)^k$ for integer $k$, but we will improve this to
		\begin{equation}\label{eq: N_n_moments}
			\mathbb{E}\#V_n(p)^k\mathbf{1}_{\{\#V_n(p) \geq k\}} \leq \begin{cases}
				\left( \Cr{C: new_main_tail} k^{\frac{\alpha-\epsilon}{2}}\left( \frac{2^n}{L(p)}\right)^{2-\alpha+\epsilon}  \right)^k & \quad\text{for } k \leq \left( \frac{2^n}{L(p)}\right)^2 \\
				\left( \Cr{C: new_main_tail} k^{\frac{\frac{5}{4} +\epsilon}{2}} \left( \frac{2^n}{L(p)}\right)^{2- \frac{5}{4} - \epsilon} \right)^k &\quad \text{for } k \geq \left( \frac{2^n}{L(p)}\right)^2
			\end{cases}
		\end{equation}
		for $n \geq 1$ and $p > 1/2$ with $L(p) \leq 2^n$. We do this by a rather involved counting argument, with many parts similar to the proof of \cite[Thm.~1.4]{K14}. Let $k \geq 1$ be an integer and write
		\begin{equation}\label{eq: moment_to_binomial}
			\mathbb{E}\#V_n(p)^k \mathbf{1}_{\{\#V_n(p) \geq k\}} \leq k^k\mathbb{E}\binom{\#V_n(p)}{k} \mathbf{1}_{\{\#V_n(p) \geq k\}} = k^k\sum_{\stackrel{V \subseteq R(n)}{\#V=k}} \mathbb{P}(V_n(p) \supseteq V).
		\end{equation}
		
		The main task in the proof of \eqref{eq: N_n_moments} is therefore to bound $\mathbb{P}(V_n(p) \supseteq V)$ in the right side of \eqref{eq: moment_to_binomial}. To do this, we fix $V = \{v_1, \dots, v_k\}$ and introduce a growing sequence of graphs $(G_i)$. Let $G_0$ be the graph $(V,\emptyset)$; that is, it has vertex set $V$ and empty edge set. We start growing an $\ell^\infty$-box at each point of $V$ at unit speed; at time $r$, we have the boxes $B(v,r)$, $v \in V$. We will stop at time $r=2^{n+1}$, at which point all boxes touch.
		
		As $r$ increases, the boxes intersect each other. Let $r_1$ be the smallest $r$ when the first pair of boxes touch. Pick one such pair of boxes in some deterministic way with centers $u_1,v_1$. We draw an edge $e_1$ between $u_1$ and $v_1$ and label it with $\ell(e_1) = r_1$, and obtain the graph $G_1$. Note that $\|u_1-v_1\|_\infty = 2r_1$. We then construct the graphs $G_i,~i=1, \dots, k-1$ inductively: once $G_i$ is constructed, we continue the growth process, and stop at time $r_{i+1}\geq r_i$ if we find a pair of vertices $u_{i+1},v_{i+1}$ with the label $\ell(e_{i+1}) = r_{i+1}$. The $r_i$'s as constructed are elements of the set $\{1/2, 1, 3/2, \dots, 2^{n+1}\}$. Also, Observation~2.1 of \cite{K14} gives an important property of the numbers $r_1, \dots, r_{k-1}$: they satisfy $r_{k-j} \leq \frac{2^{n+1}}{\lfloor \sqrt{ j }\rfloor} \text{ for } j = 1, \dots, k-1$. This implies that
		\begin{equation}\label{eq: kiss_r_constraint}
			r_{k-j} \leq 2^{n+1-\ell} \text{ for all }j \text{ with }4^\ell \leq j < 4^{\ell+1}.
		\end{equation}
		
		For $r \leq 2^{n+1}$ and $v \in R(n)$, let $\hat{\pi}_4(p;v,r)$ be the probability that the following hold:
		\begin{enumerate}
			\item $v$ is connected inside $S(n)$ to $\partial B(v,s_1)$ by two disjoint $p$-open paths, where
			\[
			s_1 = \min\{L(p),r\}.
			\]
			\item $v$ is connected inside $S(n)$ to $\partial B(v,s_2)$ by a $1/2$-closed path, where
			\[
			s_2 = \min\{\text{dist}(v,\partial S(n)), r\},
			\]
			and $\text{dist}$ is the $\ell^\infty$-distance.
			\item $v$ is connected inside $S(n)$ to $\partial B(v,s_3)$ by a $1/2$-closed path, where
			\[
			s_3 = \min\{2^n,r\}.
			\]
			\item The paths referenced in items 1-3 are alternating: open, closed, open, closed, as we proceed in a clockwise fashion around $v$.
		\end{enumerate}
	
		\begin{figure}
		\centering
		\includegraphics[width=0.8\linewidth]{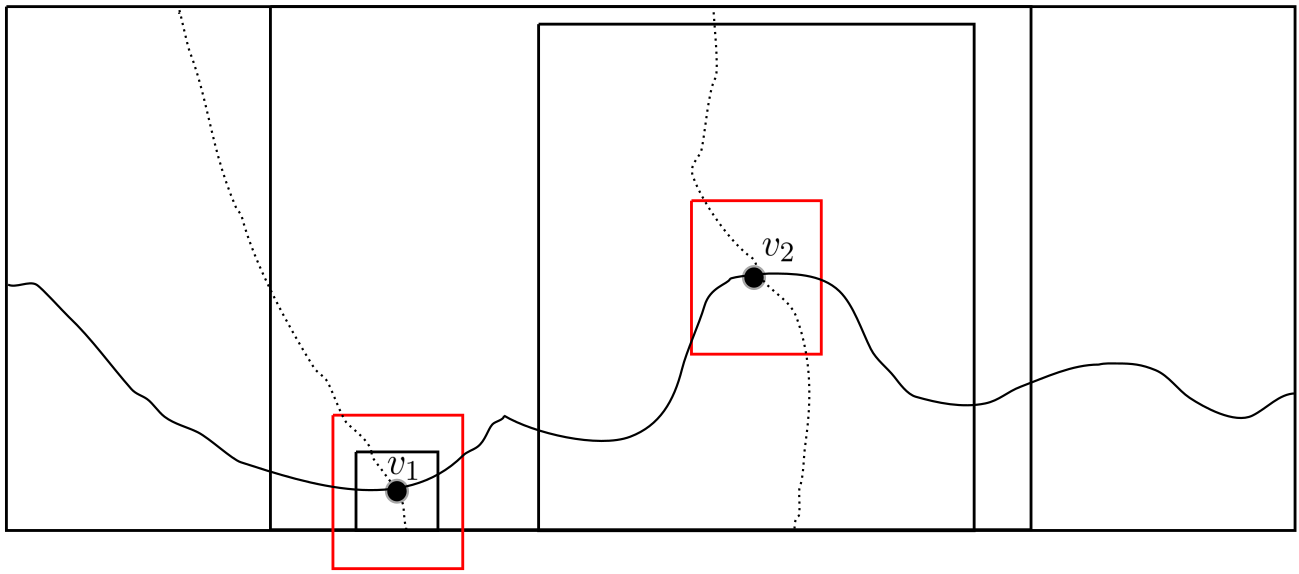}
		\caption{Depiction of the argument in the proof of Theorem \ref{thm: upper_tail_theorem}. The solid horizontal crossing of the rectangle is a $p$-open path. The dotted vertical crossings of the rectangle are $1/2$-closed paths. The largest rectangle is $S(n)$ and the second largest is $R(n)$. The red boxes have sidelength $L(p)$. The vertex $v_1$ illustrates the case when $s_2 < s_1$. The vertex $v_2$ illustrates the case when $s_1 < s_2$.}
		\label{fig: upper_tail_theorem}
	\end{figure}

		This is the same definition as that of $\hat{\pi}_4$ in \cite[p.~121]{DT19}. So, just as in Eq.~(14) there, if we put $\hat{\pi}_4(p,r) := \max_{v \in R(n)} \hat{\pi}_4(p;v,r)$, then there exists $\Cl[lgc]{C: kiss}$ such that for all integers $k,n\geq 1$ and $p > 1/2$ with $L(p) \leq 2^n$,
		\begin{equation}\label{eq: V_n_prob_bound}
			\mathbb{P}(V_n(p) \supseteq V) \leq \Cr{C: kiss} \left( p-\frac{1}{2}\right)^k \hat{\pi}_4(p,2^n) \prod_{r \in \mathcal{R}(V)} \left( \Cr{C: kiss} \hat{\pi}_4(p,r)\right).
		\end{equation}
		Here, $\mathcal{R}(V)$ is the multiset containing $r_1, \dots, r_{k-1}$ (that is, it is a set whose elements can be repeated). The argument that shows \eqref{eq: V_n_prob_bound} uses only that our probabilities $\hat{\pi}_4$ have a quasi-multiplicative property.
		
		We claim that there is $\Cl[lgc]{C: kiss_2}>0$ such that
		\begin{equation}\label{eq: pi_4_hat_bound}
			\hat{\pi}_4(p;r) \leq \Cr{C: kiss_2} \pi_4(s_1) \left(\frac{s_1}{s_3}\right)^{\alpha-\epsilon}
		\end{equation}
		for all $n \geq 1$, $r \leq 2^{n+1}$, and $p>1/2$ satisfying $L(p) \leq 2^n$. Here, $\pi_4(s_1)$ is the usual four-arm probability.
		The proof of \eqref{eq: pi_4_hat_bound} is almost the same as that of \cite[Eq.~(15)]{DT19} except there the $s_1/s_3$ term does not appear. We will follow along that proof with some details omitted. First observe that since $L(p) \leq 2^n$, and so both $s_1$ and $s_2$ are $\leq s_3$, the event defining $\hat{\pi}_4(p;v,r)$ for $v \in R(n)$ implies that $v$ is connected to distance $s_1$ by two disjoint $p$-open paths and one disjoint $1/2$-closed path, and that $v$ is connected to distance $\min\{s_1,s_2\}$ by another disjoint $1/2$-closed path, all in such a way that the paths alternate (see Figure \ref{fig: upper_tail_theorem}). Furthermore, $\partial B(v,\min\{s_1,s_2\})$ is connected to $\partial B(v,s_2)$ by two disjoint $1/2$-closed paths, and $\partial B(v,\max\{s_1,s_2\})$ is connected to $\partial B(v,s_3)$ by a $1/2$-closed path. By independence, then, we have
		\begin{equation}\label{eq: first_level_split}
			\begin{split}
				\hat{\pi}_4(p;v,r) &\leq \pi_4'(p,\min\{s_1,s_2\}) \pi_3^H(p,\min\{s_1,s_2\},s_1) \\
				&\times \rho_2(\min\{s_1,s_2\},s_2) \pi_1^H(\max\{s_1,s_2\},s_3).
			\end{split}
		\end{equation}
		Here, $\pi_4'(p,m)$ is the probability that $0$ is connected by two disjoint $p$-open paths to distance $m$ and by two disjoint $1/2$-closed paths to distance $m$ (alternating), and $\pi_3^H(p,m_1,m_2)$ is the probability that $\partial B(m_1)$ is connected to $\partial B(m_2)$ in the upper half-plane by two disjoint $p$-open paths and a $1/2$-closed path (alternating). By \cite[Lem.~6.3]{DSV09}, there exists $\Cl[lgc]{C: DSV}$ such that for all $n \geq 1$, $r \leq 2^{n+1}$, and $p>1/2$ with $L(p) \leq 2^n$,
		\[
		\pi_4'(p,\min\{s_1,s_2\}) \leq \Cr{C: DSV} \pi_4(\min\{s_1,s_2\}).
		\]
		A similar argument as in \cite[Lem.~6.3]{DSV09} also holds for half-plane three-arm (annulus) events, and we find
		\[
		\pi_3^H(p, \min\{s_1,s_2\},s_1) \leq \C[lgc] \pi_3^H(\min\{s_1,s_2\},s_1),
		\]
		where $\pi_3^H(m_1,m_2)$ is the probability that $\partial B(m_1)$ is connected by two disjoint $1/2$-open paths and a $1/2$-closed path to $\partial B(m_2)$ in such a way that all vertices $w$ on these paths satisfy $w \cdot e_2 \geq 0$. Just as in \cite[p.~112]{DT19}, we also have $\pi_3^H(m_1,m_2) \leq \C[lgc](m_1/m_2)^2$ and $\pi_4(m_1,m_2) \geq \C[smc](m_1/m_2)^2$, so using quasimultiplicativity of the four-arm probability and the two inequalities listed above, 
		for some $\Cl[lgc]{C: split}>0$, we obtain from \eqref{eq: first_level_split} 
		\[
		\hat{\pi}_4(p;v,r) \leq \Cr{C: split}\pi_4(s_1) \rho_2(\min\{s_1,s_2\},s_2) \pi_1^H(\max\{s_1,s_2\},s_3)
		\]
		for all $n \geq 1$, $r \leq 2^{n+1}$, $p>1/2$ with $L(p) \leq 2^n$, and $v \in R(n)$. This is the same as \cite[Eq.~(15)]{DT19}, except we retain the extra two factors on the right. To finish, we simply bound them using \eqref{eq: max_exponent_bound} to obtain \eqref{eq: pi_4_hat_bound}.
		
		Having established \eqref{eq: pi_4_hat_bound}, we simplify the notation by recalling that each $s_i$ depends on $r$ (as well as $n,p$) and put
		\begin{equation}\label{eq: theta_def}
			\theta(r) = \pi_4(s_1) \left( \frac{s_1}{s_3}\right)^{\alpha-\epsilon}.
		\end{equation}
		Place this in \eqref{eq: V_n_prob_bound} to obtain
		\begin{equation}\label{eq: new_V_n_prob_bound}
			\mathbb{P}(V_n(p) \supseteq V) \leq \Cl[lgc]{C: pre_multiset} \left( p-\frac{1}{2}\right)^k \theta(2^n) \prod_{r \in \mathcal{R}(V)} \left( \Cr{C: pre_multiset} \theta(r)\right).
		\end{equation}
		Furthermore, by \cite[Prop.~15]{KMS15}, for any multiset $R$ with $k-1$ elements, we have
		\begin{equation}\label{eq: ordering_bound}
			\#\{V \subset R(n) : \#V = k, \mathcal{R}(V) = R\} \leq \C[lgc] \mathcal{O}(R) 2^{2n} \prod_{r \in R} (C_{21}r),
		\end{equation}
		where $\mathcal{O}(R)$ is the number of different ways the elements of $R$ can be ordered. We use both these bounds in our previous inequalities. First, putting \eqref{eq: new_V_n_prob_bound} in \eqref{eq: moment_to_binomial}, we obtain
		\begin{align*}
			\mathbb{E}\#V_n(p)^k\mathbf{1}_{\{\#V_n(p) \geq k\}} &\leq k^k \sum_{\stackrel{V \subseteq R(n)}{\#V=k}} \left(\Cr{C: pre_multiset} \left( p-\frac{1}{2}\right)^k \theta(2^n) \prod_{r \in \mathcal{R}(V)} \left( \Cr{C: pre_multiset} \theta(r)\right) \right) \\
			&= k^k \sum_{\stackrel{V \subseteq R(n)}{\#V=k}} \left(  \sum_R \left(\Cr{C: pre_multiset} \left( p-\frac{1}{2}\right)^k \theta(2^n) \prod_{r \in R} \left( \Cr{C: pre_multiset} \theta(r)\right) \right) \mathbf{1}_{\mathcal{R}(V) = R}\right).
		\end{align*}
		Next, we restrict the sum over $R$ only to those multisets of size $k-1$ with elements from $\{1/2, 1, 3/2, \dots, 2^{n+1}\}$ such that, if we write their elements in nondecreasing order as $r_1, \dots, r_{k-1}$, then they satisfy \eqref{eq: kiss_r_constraint}. Writing $\hat{\sum_R}$ for this restricted sum, we obtain
		\[
		\mathbb{E}\#V_n(p)^k\mathbf{1}_{\{\#V_n(p) \geq k\}} \leq k^k \sum_{\stackrel{V \subseteq R(n)}{\#V=k}} \left(  \hat{\sum_R} \left(\Cr{C: pre_multiset} \left( p-\frac{1}{2}\right)^k \theta(2^n) \prod_{r \in R} \left( \Cr{C: pre_multiset} \theta(r)\right) \right) \mathbf{1}_{\mathcal{R}(V) = R}\right).
		\]
		Last, we interchange the order of summation and use \eqref{eq: ordering_bound} to find, for some $\Cl[lgc]{C: interchange}>0$,
		\begin{equation*}
			\mathbb{E}\#V_n(p)^k\mathbf{1}_{\{\#V_n(p) \geq k\}} \leq \Cr{C: interchange}^k k^k 2^{2n}\left( p-\frac{1}{2}\right)^k \theta(2^n) \hat{\sum_R} \left( \mathcal{O}(R)\prod_{r \in R} \left(r\theta(r)\right)  \right).
		\end{equation*}
		Similarly to \cite[Eq.~(3.4)]{K14}, for $j= \lfloor \log_4 k\rfloor$ and $m=k-4^j$,
		\begin{equation*}
			\begin{split}
				\mathbb{E}\#V_n(p)^k\mathbf{1}_{\{\#V_n(p) \geq k\}} &\leq \Cr{C: interchange}^k k^k2^{2n}\left( p-\frac{1}{2}\right)^k \theta(2^n) \binom{k-1}{3, 3\cdot 4, \dots, 3 \cdot 4^{j-1}, m} \\
				&\times \prod_{i=0}^{j-1} \left( \sum_{r=1}^{2^{n+1-i}} r \theta(r)\right)^{3 \cdot 4^i} \left( \sum_{r=1}^{2^{n+1-j}} r \theta(r)\right)^m.
			\end{split}
		\end{equation*}
		The multinomial coefficient is bounded as
		\[
		\binom{k-1}{3,3\cdot 4, \dots, 3 \cdot 4^{j-1}, m} = \binom{k-1}{m} \binom{k-1-m}{3 \cdot 4^{j-1}} \cdot \cdots \leq 2^{k-1 + (k-1-m) + (k-1-m-3\cdot 4^{j-1}) + \cdots} \leq \C[lgc]^k,
		\]
		so we obtain
		\begin{equation}\label{eq: moment_bound_post_counting}
			\mathbb{E}\#V_n(p)^k\mathbf{1}_{\{\#V_n(p) \geq k\}} \leq \Cl[lgc]{C: post_multinomial}^kk^k 2^{2n}\left( p-\frac{1}{2}\right)^k \theta(2^n) \prod_{i=0}^{j-1} \left( \sum_{r=1}^{2^{n+1-i}} r \theta(r)\right)^{3 \cdot 4^i} \left( \sum_{r=1}^{2^{n+1-j}} r \theta(r)\right)^m.
		\end{equation}
		
		To bound the right side, we use the fact that there is $\Cl[lgc]{C: sum_bound}>0$ such that for all $n \geq 1$ and integers $s$ with $1\leq s \leq 2^{n+1}$,
		\begin{equation}\label{eq: theta_sum_bound}
			\sum_{r=1}^s r\theta(r) \leq \Cr{C: sum_bound}s^2\theta(s).
		\end{equation}
		We split into cases to verify \eqref{eq: theta_sum_bound}. First, if $s \geq 2^n$, then the definition $\theta(r) = \pi_4(\min\{L(p),r\})\left( \frac{\min\{L(p),r\}}{\min\{2^n,r\}}\right)^{\alpha-\epsilon}$ gives
		\begin{align*}
			\sum_{r=1}^s r\theta(r) &= \sum_{r=1}^{L(p)} r \pi_4(r)  + \sum_{r=L(p)+1}^{2^n-1} r \pi_4(L(p)) \left( \frac{L(p)}{r}\right)^{\alpha-\epsilon} + \sum_{r=2^n}^s r \pi_4(L(p)) \left( \frac{L(p)}{2^n}\right)^{\alpha-\epsilon}\\
			&\leq \Cl[lgc]{C: inter_sum_bound} \left[ L(p)^2 \pi_4(L(p)) + s^2\left( \frac{L(p)}{s}\right)^{\alpha-\epsilon} \pi_4(L(p))  + s^2 \left( \frac{L(p)}{2^n}\right)^{\alpha-\epsilon} \pi_4(L(p))\right] \\
			&= \Cr{C: inter_sum_bound} s^2\theta(s) \left[ \left( \frac{L(p)}{s}\right)^2 \left( \frac{2^n}{L(p)}\right)^{\alpha-\epsilon} + \left( \frac{2^n}{s}\right)^{\alpha-\epsilon} + 1\right] \\
			&\leq 3\Cr{C: inter_sum_bound}s^2\theta(s).
		\end{align*}
		To go from the first line to the second line, we used $\sum_{r=1}^{r'} r\pi_4(r) \leq C (r')^2\pi_4(r')$, as in \cite[Lem.~3.1]{K14}.
		
		If $L(p) < s \leq 2^n-1$, then
		\begin{align*}
			\sum_{r=1}^s r\theta(r) &\leq \sum_{r=1}^{L(p)} r \pi_4(r)  + \sum_{r=L(p)+1}^s r \pi_4(L(p)) \left( \frac{L(p)}{r}\right)^{\alpha-\epsilon} \\
			&\leq \Cr{C: inter_sum_bound}\left[ L(p)^2 \pi_4(L(p)) + s^2 \left( \frac{L(p)}{s}\right)^{\alpha-\epsilon} \pi_4(L(p))\right] \leq \Cr{C: inter_sum_bound} s^2\theta(s) \left[ \left( \frac{L(p)}{s}\right)^{2-\alpha+\epsilon} + 1\right],
		\end{align*}
		which is bounded by $2\Cr{C: inter_sum_bound}s^2\theta(s)$. Last, if $1 \leq s \leq L(p)$, then
		\[
		\sum_{r=1}^s r\theta(r) = \sum_{r=1}^s r\pi_4(r) \leq \Cr{C: inter_sum_bound}s^2\pi_4(s) = \Cr{C: sum_bound}s^2\theta(s).
		\]
		This shows \eqref{eq: theta_sum_bound}.
		
		Putting \eqref{eq: theta_sum_bound} into \eqref{eq: moment_bound_post_counting} produces 
		\begin{align*}
			&\mathbb{E}\#V_n(p)^k\mathbf{1}_{\{\#V_n(p) \geq k\}}\\
			\leq~& \Cr{C: post_multinomial}^k k^k 2^{2n}\left( p-\frac{1}{2}\right)^k \left[ \theta(2^n) \prod_{i=0}^{j-1} \left(\Cr{C: sum_bound} 2^{2(n+1-i)} \theta\left( 2^{n+1-i}\right)\right)^{3 \cdot 4^i}\right] \left( \Cr{C: sum_bound} 2^{2(n+1-j)} \theta\left(2^{n+1-j}\right)\right)^m \\
			\leq~& \Cr{C: inter_sum_bound}^k k^k 2^{2nk} \left( p-\frac{1}{2}\right)^k \theta(2^n) 4^{-mj} \theta\left( 2^{n+1-j}\right)^m\left[ \prod_{i=0}^{j-1} 4^{-i3\cdot 4^i} \right] \left[ \prod_{i=0}^{j-1}\theta\left(2^{n+1-i}\right)^{3 \cdot 4^i} \right].
		\end{align*}
		We can further simplify this by using $4^{-mj}\prod_{i=0}^{j-1} 4^{-i3\cdot 4^i} \leq \C[lgc]^k k^{-k}$ and
		\[
		\theta(2^n) \theta(2^{n+1-j})^m \prod_{i=0}^{j-1} \theta(2^{n+1-i})^{3 \cdot 4^i} \leq \theta\left( \frac{2^n}{\sqrt{k}}\right)^k
		\]
		to obtain
		\[
		\mathbb{E}\#V_n(p)^k\mathbf{1}_{\{\#V_n(p) \geq k\}} \leq \left[ \Cl[lgc]{C: new_split} 2^{2n} \left( p- \frac{1}{2}\right) \theta\left( \frac{2^n}{\sqrt{k}}\right) \right]^k.
		\]
		The definition of $\theta(2^n/\sqrt{k})$ implies then that
		\[
		\mathbb{E}\#V_n(p)^k\mathbf{1}_{\{\#V_n(p) \geq k\}} \leq \begin{cases}
			\left[ \Cr{C: new_split} 2^{2n}\left( p - \frac{1}{2}\right) \pi_4(L(p)) \left( \frac{L(p)}{\frac{2^n}{\sqrt{k}}}\right)^{\alpha-\epsilon} \right]^k & \quad \text{if } k \leq \left( \frac{2^n}{L(p)}\right)^2 \\
			\left[ \Cr{C: new_split}2^{2n}\left( p - \frac{1}{2} \right) \pi_4\left( \frac{2^n}{\sqrt{k}}\right) \right]^k &\quad\text{if } k \geq \left( \frac{2^n}{L(p)}\right)^2.
		\end{cases}
		\]
		Applying the inequality $(p-1/2)L(p)^2 \pi_4(L(p)) \leq C$ from \cite[Eq.~(4.5)]{K87} yields
		\begin{equation}\label{eq: almost_end_moment_bound}
			\mathbb{E}\#V_n(p)^k\mathbf{1}_{\{\#V_n(p) \geq k\}} \leq \begin{cases}
				\left[ \Cl[lgc]{C: even_new_split} k^{\frac{\alpha-\epsilon}{2}} \left(\frac{2^n}{L(p)}\right)^{2-\alpha+\epsilon}  \right]^k & \quad \text{if } k \leq \left( \frac{2^n}{L(p)}\right)^2 \\
				\left[ \Cr{C: even_new_split} \left( \frac{2^n}{L(p)}\right)^2 \frac{\pi_4\left( \frac{2^n}{\sqrt{k}}\right)}{\pi_4(L(p))} \right]^k &\quad\text{if } k \geq \left( \frac{2^n}{L(p)}\right)^2.
			\end{cases}
		\end{equation}
		To convert this into \eqref{eq: N_n_moments}, in the bottom expression we apply quasimultiplicativity in the form $\pi_4(2^n/\sqrt{k})/\pi_4(L(p)) \leq C/\pi_4(2^n/\sqrt{k},L(p))$ and the exact value $5/4$ of the four-arm exponent from \eqref{eq: arm_exponents_1_2_4}, which gives $\pi_4(2^n/\sqrt{k},L(p)) \geq c(2^n/(\sqrt{k}L(p)))^{5/4 + \epsilon}$. Together these imply
		\[
		\frac{\pi_4\left( \frac{2^n}{\sqrt{k}}\right)}{\pi_4(L(p))} \leq \C[lgc] k^{\frac{\frac{5}{4}+\epsilon}{2}} \left( \frac{2^n}{L(p)}\right)^{-\frac{5}{4}-\epsilon}.
		\]
		Putting this in \eqref{eq: almost_end_moment_bound}, we obtain \eqref{eq: N_n_moments}.
		
		It will be more convenient to use the following consequence of \eqref{eq: N_n_moments}:
		\begin{align}
			\mathbb{E}\#V_n(p)^k &\leq k^k\mathbf{1}_{\{\#V_n(p) < k\}} + \mathbb{E}\#V_n(p)^k\mathbf{1}_{\{\#V_n(p)\geq k\}}\nonumber \\
			&\leq \begin{cases}
				\left[ \Cl[lgc]{C: use_alot} k^{\frac{\alpha-\epsilon}{2}} \left( \frac{2^n}{L(p)}\right)^{2-\alpha+\epsilon}\right]^k &\quad\text{if } k \leq \left( \frac{2^n}{L(p)}\right)^2 \\
				(\Cr{C: use_alot} k)^k & \quad \text{if } k \geq \left( \frac{2^n}{L(p)}\right)^2.
			\end{cases} \label{eq: N_n_moments_simplified}
		\end{align}
		Although we have proved this inequality for integer $k \geq 1$, by using the inequality $\mathbb{E}\#V_n(p)^k \leq \left(\mathbb{E}\#V_n(p)^{\lceil k \rceil}\right)^{k/\lceil k \rceil}$, one can extend it to real $k \geq 1$, after possibly increasing $\Cr{C: use_alot}$.
		
		Using the moment bound \eqref{eq: N_n_moments_simplified}, we will now prove the tail bound \eqref{eq: to_show_N_n} for $\#V_n(p)$. We recall that this will suffice to complete the proof of Theorem~\ref{thm: upper_tail_theorem}. To do that, we may now apply Markov's inequality to \eqref{eq: N_n_moments_simplified}. For any $\lambda>0$, and $n \geq 1$, $p>1/2$ with $L(p) \leq 2^n$,
		\begin{align*}
			&\mathbb{P}(\#V_n(p) \geq \lambda) \\
			\leq~& \min\left\{ \min_{k \leq \left( \frac{2^n}{L(p)}\right)^2} \left( \frac{\Cr{C: use_alot}\left( \frac{2^n}{L(p)}\right)^{2-\alpha+\epsilon}}{\lambda} k^{\frac{\alpha-\epsilon}{2}}\right)^k, \min_{k \geq \left( \frac{2^n}{L(p)}\right)^2} \left(\Cr{C: use_alot} \frac{k}{\lambda}\right)^k \right\}.
		\end{align*}
		We make the choice
		\[
		k = \begin{cases}
			\left( \frac{\lambda}{\Cr{C: use_alot} e \left( \frac{2^n}{L(p)}\right)^{2-\alpha + \epsilon}}\right)^{\frac{2}{\alpha-\epsilon}} &\quad\text{if } \lambda \leq \Cr{C: use_alot}e \left( \frac{2^n}{L(p)}\right)^2 \\
			\frac{\lambda}{\Cr{C: use_alot}e} &\quad\text{if } \lambda \geq \Cr{C: use_alot}e \left( \frac{2^n}{L(p)}\right)^2
		\end{cases}
		\]
		to produce the bound
		\[
		\mathbb{P}(\#V_n(p) \geq \lambda) \leq \begin{cases}
			\exp\left( - \left( \frac{\lambda}{\Cr{C: use_alot}e \left( \frac{2^n}{L(p)}\right)^{2-\alpha+\epsilon}}\right)^{\frac{2}{\alpha-\epsilon}}\right) &\quad \text{if } \lambda \leq \Cr{C: use_alot}e \left( \frac{2^n}{L(p)}\right)^2 \\
			\exp\left( -\frac{\lambda}{\Cr{C: use_alot}e}\right) &\quad \text{if } \lambda \geq \Cr{C: use_alot}e \left( \frac{2^n}{L(p)}\right)^2.
		\end{cases}
		\]
		Replacing $\lambda$ by $\lambda(2^n/L(p))^{2-\alpha+\epsilon}$, we obtain \eqref{eq: to_show_N_n}.
	\end{proof}

		\begin{figure}
		\centering
		\includegraphics[width=0.7\linewidth]{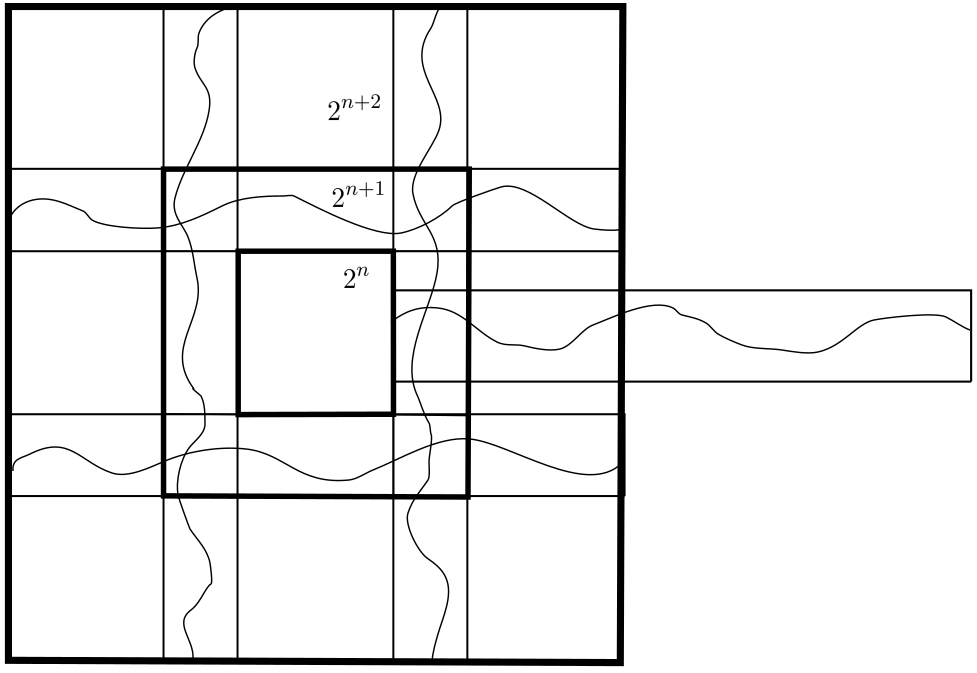}
		\caption{Depiction of the argument in the proof of Corollary \ref{cor: T_n_tail_bound}. The five rectangles with crossings are translates and rotates of the rectangle $R'(n) = [-8\cdot 2^{n-1},8\cdot 2^{n-1}]\times [-2^{n-1},2^{n-1}]$. Each curve represents a crossing in the long direction of its containing rectangle with minimal passage time.}
		\label{fig: T_n_tail_bound}
	\end{figure}

	As a result of Theorem~\ref{thm: upper_tail_theorem}, we can prove a similar tail inequality for our original variables $\mathsf{T}(n)$.
	\begin{Cor}\label{cor: T_n_tail_bound}
		Let $\epsilon \in (0,\alpha)$. There exist $\Cl[smc]{c: cor_3_2},\Cl[lgc]{C: cor_3_2}>0$ such that for any $n \geq 2$ and $p > 1/2$ with $L(p) \leq 2^{n-1}$
		\begin{align*}
			\mathbb{P}\left( \mathsf{T}(n) \geq \lambda F^{-1}(p) \left( \frac{2^n}{L(p)}\right)^{2-\alpha+\epsilon}\right) &\leq 65\exp\left( - \Cr{c: cor_3_2} \frac{2^n}{L(p)}\right) \\
			&+ \begin{cases}
				65\exp\left( -\Cr{c: cor_3_2} \lambda^{\frac{2}{\alpha-\epsilon}}\right) &\quad \text{if } \lambda \leq \Cr{C: cor_3_2} \left( \frac{2^n}{L(p)}\right)^{\alpha-\epsilon} \\
				65\exp\left( - \Cr{c: cor_3_2} \left( \frac{2^n}{L(p)}\right)^{2-\alpha+\epsilon} \lambda\right) &\quad \text{if } \lambda \geq \Cr{C: cor_3_2} \left( \frac{2^n}{L(p)}\right)^{\alpha-\epsilon}.
			\end{cases}
		\end{align*}
	\end{Cor}
	\begin{proof}
		The construction we give is similar to that in \cite[Prop.~2]{DT19}. We will build a circuit around the origin in $B(2^{n+1}) \setminus B(2^n)$ (whose passage time is an upper bound for $\mathsf{T}^{(1)}(n)$) and a path connecting $B(2^n)$ to $\partial B(2^{n+2})$ in $B(2^{n+2}) \setminus B(2^n)$ (whose passage time is an upper bound for $\mathsf{T}^{(2)}(n)$). It is possible to do this within the union of long-direction crossings of five rectangles which are translates and rotates of the rectangle $R'(n) = [-8\cdot 2^{n-1},8\cdot 2^{n-1}]\times [-2^{n-1},2^{n-1}]$ (see Figure \ref{fig: T_n_tail_bound}). By a union bound,
		\begin{equation}\label{eq: first_union}
			\mathbb{P}\left( \mathsf{T}(n) \geq \lambda F^{-1}(p) \left( \frac{2^n}{L(p)}\right)^{2-\alpha+\epsilon}\right) \leq 5 \mathbb{P}\left( T'(n) \geq \frac{\lambda}{5} F^{-1}(p) \left( \frac{2^n}{L(p)}\right)^{2-\alpha+\epsilon}\right),
		\end{equation}
		where $T'(n)$ is the minimal passage time among all paths that remain in $R'(n)$ and connect its left side to its right side.
		
		The variable $T'(n)$ is bounded using a similar geometric construction. Let $\Gamma_1, \dots, \Gamma_7$ be paths such that $\Gamma_i$ is in $[-8\cdot 2^{n-1} + 2(i-1)2^{n-1}, -8\cdot 2^{n-1}+ 2(i+1)2^{n-1}] \times [-2^{n-1},2^{n-1}]$, connects the left side of the rectangle to its right side, and has minimal passage time among all such paths. Let $\hat{\Gamma}_1, \dots, \hat{\Gamma}_6$ be paths such that $\hat{\Gamma}_i$ is in $[-8\cdot 2^{n-1} + 2i2^{n-1}, -8\cdot2^{n-1}+2(i+1)2^{n-1}] \times [-2^{n-1},3\cdot 2^{n-1}]$, connects the top side of the rectangle to the bottom side, and has minimal passage time among all such paths. By planarity, there is a path remaining in $[-8\cdot 2^{n-1},8\cdot 2^{n-1}]\times[-2^{n-1},2^{n-1}]$ which starts on the left side of this rectangle, ends on the right, and is contained in the union of the $\Gamma_i$'s and the $\hat{\Gamma}_i$'s. We therefore find, for $T(n-1)$ defined as below \eqref{eq: max_exponent_bound},
		\[
		\mathbb{P}\left( T'(n) \geq \frac{\lambda}{5} F^{-1}(p) \left( \frac{2^n}{L(p)}\right)^{2-\alpha+\epsilon}\right) \leq 13\mathbb{P}\left( T(n-1) \geq \frac{\lambda}{65} F^{-1}(p) \left( \frac{2^n}{L(p)}\right)^{2-\alpha+\epsilon}\right).
		\]
		We apply Theorem~\ref{thm: upper_tail_theorem} to the right side and combine this with \eqref{eq: first_union} to complete the proof.
	\end{proof}

	\subsubsection{Ruling out exceptional times}\label{sec: ruling_out}
	Now that we have the tail bound for $\mathsf{T}(n)$ in Corollary~\ref{cor: T_n_tail_bound}, we can return to the proof of Theorem~\ref{thm: other_side}. Suppose that $\sum_{k=2}^\infty k^{7/8}a_k < \infty$. Recall that if $\mathsf{T}_t(n)$ is the variable $\mathsf{T}(n)$ evaluated in the configuration $(\tau_v(t))$, it suffices by \eqref{eq: t_max_implication} to prove that
	\begin{equation}\label{eq: finally_to_show_other_side}
		\text{a.s., for all } t \in [0,1],~ \sum_{n=1}^\infty \mathsf{T}_t(n) < \infty.
	\end{equation}
	To do this, we will bound partial sums of $\mathsf{T}_t(n)$ by partial sums of $\sum_k k^{7/8} a_k$. We will need a simple tool that helps us bound $\mathsf{T}_t(n)$ for all $t$ simultaneously.
	\begin{Lemma}\label{lem: poisson_lemma}
		For each $n$, let $A_n$ be an event depending on the weights $(\tau_v)$ for $v \in B(2^n)$. If $\mathbb{P}(A_n) \leq 18^{-n}$ for all large $n$, then
		\[
		\sum_{n=1}^\infty \mathbb{P}\left(A_n \text{ occurs in } (\tau_v(t)) \text{ for some } t \in [0,1]\right) <\infty.
		\]
	\end{Lemma}
	\noindent
	By the Borel-Cantelli Lemma, the above lemma implies that a.s., for all large $n$, $A_n^c$ occurs for all $t \in [0,1]$ simultaneously.
	\begin{proof}
		There is no harm in assuming $\mathbb{P}(A_n) \leq 18^{-n}$ for all $n$. To prove the lemma, we will use a union bound over a fine subdivision of times. For this purpose, we define times $0=t_0^{(n)} < t_1^{(n)} < \dots < t_{p_n}^{(n)} = 1$ with
		\[
		t_i^{(n)} - t_{i-1}^{(n)} \leq 17^{-n} \text{ and } p_n \leq  17^n +1.
		\]
		Note that if $A_n$ occurs at some time $t \in [0,1]$ but not at any $t_i^{(n)}$, then a.s., at least one interval $\bigg( t_{i-1}^{(n)}, t_i^{(n)}\bigg]$ must contain at least two times at which vertices in $B(2^n)$ resample their values. Recalling that $(\mathfrak{s}_v(t))_{t \geq 0}$ is the Poisson process associated with vertex $v$, and setting $\mathfrak{s}(t;n) = \sum_{v \in B(2^n)} \mathfrak{s}_v(t)$ (which is a Poisson process of rate $r(n) = (2^n+1)^2$), then, we obtain
		\begin{align*}
			&\mathbb{P}(A_n \text{ occurs in }(\tau_v(t))\text{ for some } t \in [0,1]) \\
			\leq~& \mathbb{P}(A_n \text{ occurs in }(\tau_v(t_i^{(n)})) \text{ for some } i)  \\
			+~& \mathbb{P}\left(\mathfrak{s}(t_i^{(n)};n) > \mathfrak{s}(t_{i-1}^{(n)};n) +1 \text{ for some }i\right) \\
			\leq~& (p_n+1) \mathbb{P}(A_n) + p_n \mathbb{P}(\mathfrak{s}(17^{-n};n) \geq 2) \\
			\leq~& \left(17^n + 2\right) \left(18^{-n} + \left( 1 - e^{-17^{-n}r(n)} - 17^{-n}r(n)e^{-17^{-n}r(n)}\right)\right) \\
			\leq~&  \left(17^n + 2\right) \left(18^{-n} + (17^{-n}r(n))^2 \right).
		\end{align*}
		Because $r(n) \leq 4\cdot4^n$, the bound is summable and this completes the proof.
	\end{proof}
	
	Our next task is to use Lemma~\ref{lem: poisson_lemma} in conjunction with the tail bound Corollary~\ref{cor: T_n_tail_bound} to truncate the variables $\mathsf{T}_t(n)$. To do this, we first reformulate the inequality of the corollary in a more useful way. Use $\lambda = \Cr{C: cor_3_2} \left( \frac{2^n}{L(p)}\right)^{\frac{\alpha-\epsilon}{2}}$ to obtain for $n \geq 2$ and $p > 1/2$ with $L(p) \leq 2^{n-1}$
	\begin{equation*}\label{eq: hamburger_pastaroni}
		\mathbb{P}\left( \mathsf{T}(n) \geq \Cr{C: cor_3_2} F^{-1}(p) \left( \frac{2^n}{L(p)}\right)^{2-\frac{\alpha-\epsilon}{2}}\right) \leq \Cl[lgc]{C: pastaroni} \exp\left( - \C[smc] \frac{2^n}{L(p)}\right).
	\end{equation*}
	If we set $p = p_{\left\lceil 2^n/u\right\rceil}$ in this equation and write 
	\begin{equation}\label{eq: eta_def}
		\eta = 2-\frac{\alpha-\epsilon}{2}
	\end{equation}
	for simplicity, we obtain from \eqref{eq: approx_inverse} that
	\[
	\mathbb{P}\left( \mathsf{T}(n) \geq \Cl[smc]{c: stretched} u^\eta F^{-1}\left( p_{\left\lceil \frac{2^n}{u}\right\rceil}\right) \right) \leq \Cr{C: pastaroni} \exp\left( -\Cr{c: stretched} u\right) \text{ for } u \in \left[ 4,2^n\right].
	\]
	Substituting $u^{1/\eta}$ for $u$, this becomes 
	\[
	\mathbb{P}\left( \mathsf{T}(n) \geq \Cr{c: stretched} u F^{-1}\left( p_{\left\lceil \frac{2^n}{u^{\frac{1}{\eta}}}\right\rceil}\right) \right) \leq \Cr{C: pastaroni} \exp\left( -\Cr{c: stretched} u^{\frac{1}{\eta}}\right) \text{ for } u \in \left[ 4^\eta,2^{n\eta}\right].
	\]
	As long as $u \leq (4/3)^{n\eta}$, we have $1.5^n \leq 2^n/u^{1/\eta}$, so we obtain for $n \geq 2$
	\begin{equation}\label{eq: new_tail_version}
		\mathbb{P}\left( \mathsf{T}(n) \geq \Cr{c: stretched} u F^{-1}\left( p_{\left\lceil 1.5^n\right\rceil}\right) \right) \leq \Cr{C: pastaroni} \exp\left( - \Cr{c: stretched} u^{\frac{1}{\eta}}\right)\text{ for } u \in \left[ 4^\eta, \left( \frac{4}{3}\right)^{n\eta}\right].
	\end{equation}
	
	Equation \eqref{eq: new_tail_version} is our more useful version of the bound of Corollary~\ref{cor: T_n_tail_bound}. Because it only applies for $u \leq (4/3)^{n\eta}$, we truncate our variables at this level. Set
	\[
	X_{n,t} = \frac{\mathsf{T}_t(n)}{\Cr{c: stretched} F^{-1}\left( p_{\left\lceil 1.5^n\right\rceil}\right)} \text{ and } Y_{n,t} = X_{n,t} \mathbf{1}_{\{X_{n,t} \leq \left( \frac{4}{3}\right)^{n\eta}\}}.
	\]
	For future reference, we record that
	\begin{equation}\label{eq: Y_bound}
		\mathbb{P}\left( Y_{n,0} \geq u\right) \leq \Cr{C: pastaroni} \exp\left( - \Cr{c: stretched} u^{\frac{1}{\eta}}\right) \text{ for } u \geq 0 \text{ and } n \geq 2.
	\end{equation}
	We now claim
	\begin{equation}\label{eq: truncation_holds}
		\sum_{n=1}^\infty \mathbb{P}(X_{n,t} \neq Y_{n,t} \text{ for some } t \in [0,1]) < \infty.
	\end{equation}
	To prove this, we apply Lemma~\ref{lem: poisson_lemma}. For some $\Cl[smc]{c: more_stretched}>0$,
	\begin{align*}
		\mathbb{P}(X_{n,0} \neq Y_{n,0}) &\leq \mathbb{P}\left(\mathsf{T}(n) \geq \Cr{c: stretched} F^{-1}\left( p_{\left\lceil 1.5^n\right\rceil}\right) \left( \frac{4}{3}\right)^{\eta n}\right) \\
		&\leq \mathbb{P}\left( \mathsf{T}(n) \geq \Cr{c: more_stretched} \left( \frac{4}{3}\right)^{\frac{\alpha-\epsilon}{2} n}  F^{-1}\left( p_{\left\lceil 1.5^n\right\rceil}\right) \left( \frac{2^n}{L\left( p_{\left\lceil 1.5^n\right\rceil}\right)}\right)^{2-\alpha+\epsilon}
		\right).
	\end{align*}
	We use Corollary~\ref{cor: T_n_tail_bound} to get
	\begin{equation}\label{eq: X_n_Y_n}
		\mathbb{P}(X_{n,0} \neq Y_{n,0}) \leq 65 \exp\left( -\Cr{c: cor_3_2} \frac{2^n}{L\left( p_{\left\lceil 1.5^n\right\rceil}\right)}\right) + 65 \exp \left( -\C[smc] \left( \frac{4}{3}\right)^{n}\right).
	\end{equation}
	Because this is $\leq 18^{-(n+2)}$ for all large $n$, and the variables $X_{n,0},Y_{n,0}$ depend only on weights associated to vertices in $B(2^{n+2})$, Lemma~\ref{lem: poisson_lemma} implies \eqref{eq: truncation_holds}.
	
	Because of \eqref{eq: truncation_holds}, to prove \eqref{eq: finally_to_show_other_side}, we are reduced to showing that
	\begin{equation}\label{eq: last_to_show}
		\text{a.s., for all } t \in [0,1],~ \sum_{n=1}^\infty \left[ F^{-1}\left( p_{\left\lceil 1.5^n\right\rceil}\right) Y_{n,t}\right] <\infty.
	\end{equation}
	To do this, we apply Lemma~\ref{lem: sum_by_parts} with
	$\mathsf{a}_k = Y_{k,t}$ and $\mathsf{b}_k = F^{-1}\left( p_{\left\lceil 1.5^k\right\rceil}\right)$, so we must bound the partial sum $\sum_{k=1}^n Y_{k,t}$. We will show that for some $\Cl[lgc]{C: ind_sum}>0$,
	\begin{equation}\label{eq: sum_bound_claim}
		\mathbb{P}\left( Y_{2,0} + \dots + Y_{n,0} \geq \Cr{C: ind_sum}n^\eta\right) \leq 18^{-(n+2)} \text{ for all large } n.
	\end{equation}
	Because the event $\{Y_{2,0} + \dots + Y_{n,0} \geq \Cr{C: ind_sum}n^\eta\}$ depends only on weights associated to vertices in $B(2^{n+2})$, Lemma~\ref{lem: poisson_lemma} will then imply that a.s.,
	\[
	C_t := \sup_{n \geq 1} \frac{1}{n^\eta} \sum_{k=1}^n Y_{k,t} < \infty \text{ for all } t \in [0,1].
	\]
	That is, we will obtain the following partial sum bound: a.s.,
	\begin{equation}\label{eq: partial_sum_bound}
		\sum_{k=1}^n Y_{k,t} \leq C_t n^\eta \leq \Cl[lgc]{C: burrito} C_t \sum_{k=1}^n k^{\eta-1} \text{ for all } t \in [0,1] \text{ and } n \geq 1.
	\end{equation}
	
	To prove \eqref{eq: sum_bound_claim}, we give the following lemma:
	\begin{Lemma}\label{lem: xi_lemma}
		Let $\xi_j, j=1, 2, \dots$ be i.i.d.~random variables with $\mathbb{E}\xi_1=0$ and $\mathrm{Var}~\xi_1 = 1$. Assume that $\xi_1$ has a continuous distribution with density $p(x)$ such that $p(x) \sim e^{-x^{1-\delta}}$ as $x \to \infty$ for some $\delta \in (0,1)$ and $\mathbb{E}|\xi_1|^k<\infty$ for some $k \geq \lceil 1/\delta-2 \rceil +3$. If $(x_n)$ is a sequence with $x_n/n^{1/(2\delta)} \to \infty$, then
		\begin{equation}\label{eq: xi_lemma}
			\frac{\mathbb{P}(\xi_1 + \dots + \xi_n > x_n)}{n\mathbb{P}(\xi_1 > x_n)} \to 1 \text{ as } n \to \infty.
		\end{equation}
	\end{Lemma}
	\begin{proof}
		This is \cite[Thm.~3]{N69} with the condition $p(x) \sim e^{-|x|^{1-\delta}}$ as $|x|\to\infty$ in \cite[Eq.~(2)]{N69} replaced by the stated condition, which appears in item 1 of the footnote of \cite[p.~51]{N69}.
	\end{proof}
	Below, we will apply the lemma with $\delta = 1-1/\eta$ and $x_n$ of order $n^{\eta} = n^{1/(1-\delta)}$. Then $x_n/n^{1/(2\delta)}$ will be of order $n^{1/(1-\delta) - 1/(2\delta)} \to \infty$ so long as $\delta>1/3$, and this holds because $\eta \geq 11/6$. We conclude that for random variables as in the lemma with $p(x) \sim e^{-x^{1/\eta}}$ and $x_n$ of order $n^{\eta}$, \eqref{eq: xi_lemma} holds.

	To use Lemma~\ref{lem: xi_lemma} to prove \eqref{eq: sum_bound_claim}, we first observe that although the sequence $(Y_{n,0})_{n \geq 2}$ is not independent, it is only 1-dependent, so $(Y_{2n,0})_{n \geq 1}$ consists of independent random variables, as does $(Y_{2n-1,0})_{n \geq 2}$. By splitting into contributions from these subsums, we may assume that the original sequence $(Y_{n,0})_{n \geq 2}$ is independent. Assuming this, by the bound \eqref{eq: Y_bound}, $(Y_{n,0})_{n \geq 2}$ is stochastically dominated by the sequence $(\Cl[lgc]{C: stoch_dom} (1+\xi_n))_{n \geq 2}$, where $\Cr{C: stoch_dom}$ is a positive constant, and the $\xi_i$'s are as in Lemma~\ref{lem: xi_lemma}. Therefore for large $n$, the left side of \eqref{eq: sum_bound_claim} is no larger than
	\begin{align*}
		\mathbb{P}(Y_{2,0} + \dots + Y_{n,0} > \Cr{C: ind_sum}n^\eta - 1) &\leq \mathbb{P}(\xi_2 + \dots + \xi_n > \Cr{C: stoch_dom}^{-1}(\Cr{C: ind_sum}n^\eta -1)-n) \\
		&\leq \mathbb{P}\left(\xi_1 + \dots + \xi_n > \frac{\Cr{C: ind_sum}}{2\Cr{C: stoch_dom}}n^\eta\right).
	\end{align*}
	Applying Lemma~\ref{lem: xi_lemma} with $x_n = n^\eta \Cr{C: ind_sum}/(2\Cr{C: stoch_dom})$, for large $n$ we obtain the upper bound
	\[
	2n\mathbb{P}\left( \xi_1 \geq \frac{\Cr{C: ind_sum}}{2\Cr{C: stoch_dom}}n^\eta\right) \leq 18^{-(n+2)},
	\]
	so long as $C_{39}$ is large enough. This gives \eqref{eq: sum_bound_claim}.

	We now use the consequence \eqref{eq: partial_sum_bound} of \eqref{eq: sum_bound_claim} to apply Lemma~\ref{lem: sum_by_parts} with $\mathsf{a}_k = Y_{k,t}, \mathsf{b}_k = F^{-1}\left( p_{\left\lceil 1.5^k\right\rceil}\right)$, and $\mathsf{c}_k = \Cr{C: ind_sum}C_t k^{\eta-1}$. We conclude that a.s., for all $t \in [0,1]$ and $n \geq 1$,
	\begin{equation}\label{eq: seriously_last}
		\sum_{k=1}^n \left[ F^{-1}\left( p_{\left\lceil 1.5^k\right\rceil}\right) Y_{k,t}\right] \leq \Cr{C: ind_sum}C_t \sum_{k=1}^n \left[ F^{-1}\left( p_{\left\lceil 1.5^k\right\rceil}\right) k^{\eta-1}\right].
	\end{equation}
	Finally \eqref{eq: p_n_bounds} gives for some $\Cl[smc]{c: nachos_bellegrande}>0$
	\[
	F^{-1}\left( p_{\left\lceil 1.5^k\right\rceil}\right) \leq F^{-1}\left( \frac{1}{2} + \frac{1}{2^{\lceil \Cr{c: nachos_bellegrande}k\rceil}}\right) = a_{\lceil \Cr{c: nachos_bellegrande}k\rceil}
	\]
	for all large $k$. Comparing to \eqref{eq: seriously_last}, we see that \eqref{eq: last_to_show} holds as long as $\sum_k (a_{\lceil \Cr{c: nachos_bellegrande}k\rceil}\wedge 1) k^{\eta-1}< \infty$. Recall that our assumption is that $\sum a_k k^{7/8} < \infty$. Because $\eta - 1 = 1-(\alpha-\epsilon)/2$ and $\alpha > 1/4$, there is $\epsilon$ small enough that $\eta-1 < 7/8$. This implies \eqref{eq: last_to_show} and completes the proof.

	\subsection{Proof of Theorem~\ref{thm: Hausdorff}}\label{sec: Hausdorff}
	
	In this section we suppose that $F$ satisfies \eqref{eq: critical_def} and \eqref{eq: infinite_sum}. To prove that the Hausdorff dimension equals $31/36$ a.s., we will prove that
	\begin{equation}\label{eq: Hausdorff_lower}
		\text{dim}_\text{H}(\{t \geq 0 : \rho_t < \infty\}) \geq \frac{31}{36} \text{ a.s.},
	\end{equation}
	and 
	\begin{equation}\label{eq: Hausdorff_upper}
		\text{dim}_\text{H}(\{t \geq 0 : \rho_t < \infty\}) \leq \frac{31}{36} \text{ a.s.}
	\end{equation}

	The lower bound, \eqref{eq: Hausdorff_lower}, follows quickly from the proof of \cite[Thm.~1.4(1)]{GPS10}. Specifically, \cite[p.~92]{GPS10} states ``The above discussion therefore gives $\text{dim}_\text{H}(\mathcal{E})\geq 31/36$ a.s.'' Here, $\mathcal{E}$ is defined at the bottom of \cite[p.~91]{GPS10} as ``the set of exceptional times $t \in [0,\infty)$ for the event that the origin is in an infinite open cluster.'' The term ``infinite open cluster'' corresponds here to an infinite connected set of vertices with weight zero, so a.s. $\text{dim}_\text{H}(\{t \geq 0 : \rho_t = 0\}) = 31/36$ and
	\[
	\text{dim}_\text{H}(\{t \geq 0 : \rho_t < \infty\}) \geq \text{dim}_\text{H}(\{t \geq 0 : \rho_t = 0\}) = \frac{31}{36} \text{ a.s.}
	\]
	In other words, \eqref{eq: Hausdorff_lower} holds.
	
	To prove \eqref{eq: Hausdorff_upper}, we note that by time stationarity and countable stability, it suffices to prove that
	\begin{equation}\label{eq: reduced_hausdorff_upper_bound}
		\text{dim}_\text{H}(\{t \in [0,1] : \rho_t < \infty\}) \leq \frac{31}{36} \text{ a.s.}
	\end{equation}
	We will build an explicit cover using an integer $L>1$, which we take large later in the proof. For $k \geq 1$, let $A_k(t)$ denote the event that there is a circuit around 0 in the annulus $\text{Ann}(L^{k-1},L^k)$ such that each of its vertices $v$ satisfy $\omega_v(t) \geq q_k > 1/2,$ where
	\[
	q_k = \frac{1}{2} + \frac{1}{2}\left( p_{L^k} - \frac{1}{2}\right).
	\]
	We observe that every infinite, self-avoiding path starting at 0 must cross each of these annuli, and so by \eqref{eq: rho_equivalent},
	\begin{equation}\label{eq: rho_lower_bound}
		\sum_{k=1}^\infty F^{-1}(q_k) \mathbf{1}_{A_k(t)} \leq \rho_t.
	\end{equation}
	We claim that for any $t$ and any outcome,
	\begin{equation}\label{eq: density}
		\rho_t < \infty \Rightarrow \liminf_{n \to \infty} \frac{1}{n} \sum_{k=1}^n \mathbf{1}_{A_k(t)} = 0.
	\end{equation}
	We argue by contrapositive. Assume that for some $\Cl[smc]{c: contradiction}>0$, we have $\sum_{k=1}^n \mathbf{1}_{A_k(t)} \geq \Cr{c: contradiction}n$ if $n$ is larger than some $n_0$. Then set $\mathsf{c}_k = \mathbf{1}_{A_k(t)}$ if $k \geq n_0$ and 1 otherwise, so that $\sum_{k=1}^n \mathsf{c}_k \geq \Cr{c: contradiction}n$ for all $n \geq 1$. Lemma~\ref{lem: sum_by_parts} with $\mathsf{a}_k = \Cr{c: contradiction}$ and $\mathsf{b}_k = F^{-1}(q_k)$ implies that for all $n \geq 1$,
	\begin{equation}\label{eq: sum_lower_bound}
		\Cr{c: contradiction} \sum_{k=1}^n F^{-1}(q_k) \leq \sum_{k=1}^n \mathsf{c}_k F^{-1}(q_k).
	\end{equation}
	By \eqref{eq: p_n_bounds}, we have for some $\epsilon_2>0$, $F^{-1}(q_k)   \geq a_{\lfloor \epsilon_2 k\rfloor}$ if $k$ is large. Because we have assumed \eqref{eq: infinite_sum}, this is not summable, and therefore \eqref{eq: sum_lower_bound} implies that $\sum \mathsf{c}_k F^{-1}(q_k)$ diverges. By \eqref{eq: rho_lower_bound}, $\rho_t = \infty$, and this shows \eqref{eq: density}.
	
	To bound the dimension of the set of exceptional times, we will cover the set of $t$ for which the right side of \eqref{eq: density} holds. Our cover will consist of intervals of different sizes. Let 
	\begin{equation}\label{eq: delta_k_def}
		\Delta_k = \frac{p_{L^k}-\frac{1}{2}}{2}
	\end{equation}
	and let $\Pi_k$ be the collection of half-open intervals $I$ of the form
	\[
	I = [ih,(i+1)h) \text{ for } i=0, \dots, \lceil \Delta_k^{-1}\rceil -1
	\]
	and $h= 1/\lceil 1/\Delta_k\rceil$. To decide which of these intervals to include in our covering, we define auxiliary weights $\sigma_v^k(t)$ for vertices $v$ and $t \in [0,1)$, $k \geq 1$ by $\sigma_v^k(t) = 1$ if $\omega_v(t) \geq q_k$ and $\mathfrak{s}_v$ does not increment on the interval $I$, where $I$ is the unique element of $\Pi_k$ that contains $t$. Otherwise, we set $\sigma_v^k(t) = 0$. Let $B_k(t)$ be the event that there is a path connecting the inner and outer boundaries of $\text{Ann}(L^{k-1},L^k)$ whose vertices $v$ satisfy $\sigma_v^k(t) = 0$. We observe that $A_k^c(t) \subset B_k(t) \text{ for } t \in [0,1) \text{ and } k \geq 1$. Indeed, if $B_k(t)$ does not occur, then by duality, there is a circuit around 0 in $\text{Ann}(L^{k-1},L^k)$ whose vertices $v$ satisfy $\sigma_v^k(t) = 1$, and therefore $\omega_v(t) \geq q_k$, implying occurrence of $A_k(t)$. Because of this,
	\begin{equation}\label{eq: w_implication}
		\liminf_{n \to \infty} \frac{1}{n} \sum_{k=1}^n \mathbf{1}_{A_k(t)} = 0 \Rightarrow \limsup_{n \to \infty} \overline{W}_n(t) = 1,
	\end{equation}
	where $\overline{W}_n(t) = (1/n) \sum_{k=1}^n \mathbf{1}_{B_k(t)}$. Therefore if $x \in (0,1)$ is fixed, and the left side of \eqref{eq: w_implication} holds for some $t \in [0,1)$, then $\overline{W}_n(t) \geq x$ for infinitely many $n$. This motivates our covering: we set
	\[
	\mathfrak{C}_N(x) = \left\{ I \in \Pi_n : n \geq N, \overline{W}_n(t) \geq x \text{ for some } t \in I\right\},
	\]
	so that by \eqref{eq: density} and \eqref{eq: w_implication},
	\[
	\text{for any } N \text{ and } x \in (0,1),~\{ t \in [0,1) : \rho_t < \infty\} \subset \bigcup_{I \in \mathfrak{C}_N(x)} I .
	\]
	Using the definition of Hausdorff outer measure in \eqref{eq: hausdorff_definition} and the fact that $\{1\}$ has diameter 0, for any $\alpha \in (0,1)$,
	\begin{equation}\label{eq: cover_containment}
		\mathcal{H}_\alpha\left( \{t \in [0,1] : \rho_t < \infty\}\right) \leq \liminf_{N \to \infty} \sum_{I \in \mathfrak{C}_N(x)} (\text{diam}~I)^\alpha.
	\end{equation}
	
	To bound the right side of \eqref{eq: cover_containment}, we note that the weights $\sigma_v^k(t)$ are constant for $t$ in any interval in $\Pi_k$, so
	\[
	\sum_{I \in \mathfrak{C}_N(x)} \left( \text{diam}~I\right)^\alpha = \sum_{n \geq N} \sum_{i=0}^{\lceil \Delta_n^{-1}\rceil -1} \lceil\Delta_n^{-1}\rceil^{-\alpha} \mathbf{1}_{\{\overline{W}_n(i\lceil \Delta_n^{-1}\rceil^{-1}) \geq x\}},
	\]
	and
	\begin{equation}\label{eq: expected_covering_bound}
		\mathbb{E}\sum_{I \in \mathfrak{C}_N(x)} \left( \text{diam}~I\right)^\alpha = \sum_{n \geq N} \lceil \Delta_n^{-1}\rceil^{1-\alpha} \mathbb{P}(\overline{W}_n(0) \geq x).
	\end{equation}
	For any vertex $v$, and integer $k \geq 1$,
	\[
	\mathbb{P}(\sigma_v^k(0) = 0) = 1-(1-q_k)e^{-\lceil \Delta_k^{-1}\rceil^{-1}} \leq q_k + \lceil \Delta_k^{-1}\rceil^{-1} \leq p_{L^k},
	\]
	so from \eqref{eq: near_critical_arm_event_comparability}, for all $L,k \geq 1$,
	\[
	\mathbb{P}(B_k(0)) \leq \Cr{C: change_p} \pi_1(L^{k-1},L^k).
	\]
	Using quasimultiplicativity from \eqref{eq: quasimultiplicativity} and $\alpha_1 = 5/48$ from \eqref{eq: arm_exponents_1_2_4}, given any $\delta_1 < 5/48$, there exists $\Cl[lgc]{C: local_arm}>0$ such that for all $L,k$, $\pi_1(L^{k-1},L^k) \leq \Cr{C: local_arm}L^{-\delta_1}$. Therefore, calling 
	\begin{equation}\label{eq: p_def}
		p = p(L) = \Cr{C: change_p}\Cr{C: local_arm}L^{-\delta_1},
	\end{equation}
	we have $\mathbb{P}(B_k(0)) \leq p \text{ for all } k,L \geq 1$, and so
	\begin{equation}\label{eq: W_domination}
		\mathbb{P}\left( \overline{W}_n(0) \geq x\right) \leq \mathbb{P}(X_1 + \dots + X_n \geq nx),
	\end{equation}
	where $X_1, \dots, X_n$ are i.i.d.~random variables having the Bernoulli distribution with parameter $p$. Standard large deviation results dictate that if we define the rate function 
	\begin{equation}\label{eq: rate_function}
		I(x) = I(x,p) = x\log \frac{x}{p} + (1-x)\log \frac{1-x}{1-p},
	\end{equation}
	then for $x \geq p$,
	\[
	\mathbb{P}(X_1 + \dots + X_n \geq nx) \leq e^{-I(x)n}.
	\]
	Putting this in \eqref{eq: W_domination} and then back in \eqref{eq: expected_covering_bound}, we obtain
	\[
	\mathbb{E}\sum_{I \in \mathfrak{C}_N(x)} \left( \text{diam}~I\right)^\alpha \leq 2^{1-\alpha} \sum_{n \geq N} \Delta_n^{\alpha-1} e^{-I(x)n} = 2^{1-\alpha}\sum_{n \geq N} \left( \frac{p_{L^n}-\frac{1}{2}}{2}\right)^{\alpha-1} e^{-I(x)n}.
	\]
	By \eqref{eq: p_n_bounds}, if $\epsilon_1>3/4$, the sum is at most $4^{1-\alpha} \sum_{n \geq N} e^{-nI(x)}L^{n \epsilon_1 (1-\alpha)}$ for large $N$. If $I(x) > \epsilon_1(1-\alpha)\log L$, then this converges to 0 as $N \to \infty$, so we apply Fatou's lemma in \eqref{eq: cover_containment} to deduce that
	\[
	\alpha > 1 - \frac{I(x)}{\epsilon_1 \log L} \Rightarrow \mathbb{E}\mathcal{H}_\alpha(\{t \in [0,1] : \rho_t < \infty\}) = 0.
	\]
	If we let $x\uparrow 1$, we see in \eqref{eq: rate_function} that $I(x) \to \log(1/p)$, and so this implies
	\[
	\alpha > 1- \frac{\log(1/p)}{\epsilon_1 \log L} \Rightarrow \mathcal{H}_\alpha(\{t \in [0,1] : \rho_t < \infty\}) = 0 \text{ a.s.}
	\]
	The definition of $p$ in \eqref{eq: p_def} along with the definition of Hausdorff dimension gives
	\[
	\textnormal{dim}_\textnormal{H}(\{t \in [0,1] : \rho_t < \infty\}) \leq  1 + \frac{\log \Cr{C: change_p}\Cr{C: local_arm} - \delta_1 \log L}{\epsilon_1 \log L} \text{ a.s.}
	\]
	Letting $L \to \infty$, we obtain an upper bound of $1-\delta_1/\epsilon_1$. If we take $\delta_1 \to 5/48$ and $\epsilon_1 \to 3/4$, this shows \eqref{eq: reduced_hausdorff_upper_bound} and completes the proof of \eqref{eq: Hausdorff_upper}. Since we have established both the lower bound \eqref{eq: Hausdorff_lower} and the upper bound \eqref{eq: Hausdorff_upper}, this finishes the proof of Theorem~\ref{thm: Hausdorff}.

	\subsection{Proof of Theorem~\ref{thm: Minkowski}}\label{sec: Minkowski}
	
	The proof of Theorem~\ref{thm: Minkowski} will be split into two parts. Throughout, we assume that $F$ satisfies \eqref{eq: critical_def} and \eqref{eq: infinite_sum}.
	
	\subsubsection{Theorem~\ref{thm: Minkowski}(1)}
	
	We begin the lower bound in the first item of Theorem~\ref{thm: Minkowski}. It again uses \cite{GPS10}, but now the fact that on the event $\{\exists t \in [0,1] : \rho_t = 0\}$, a.s.~one has $\text{dim}_\text{H}(\{t \in [0,1] : \rho_t = 0\}) \geq 31/36.$ This follows from the discussion at the bottom of \cite[p.~91]{GPS10} along with the statement on \cite[p.~92]{GPS10} that $\sup_R M_\gamma(R) < \infty$ for any $\gamma < 31/36$. Using time stationarity, we deduce the same claim for general intervals: if $s\geq 1$, then on the event $\{\exists t \in [0,s] : \rho_t = 0\}$, one has $\text{dim}_\text{H}(\{t \in [0,s] : \rho_t = 0\}) \geq 31/36$.
	
	Because $\{t \geq 0 : \rho_t = 0\}$ is a.s.~nonempty (its Hausdorff dimension is a.s.~31/36), we can, given $\delta>0$, choose $s\geq 1$ such that
	\[
	\mathbb{P}\left( \{t \in [0,s] : \rho_t = 0\} \neq \emptyset\right) > 1-\delta.
	\]
	Therefore for any $x \in [0,\infty)$,
	\begin{align*}
		&\mathbb{P}\left( \textnormal{dim}^\textnormal{M}(\{t \in [0,s] : \rho_t \leq x\}) \geq \frac{31}{36}\right) \\
		\geq~& (1-\delta) \mathbb{P}\left( \textnormal{dim}_\textnormal{H}\left( \{t \in [0,s] : \rho_t = 0\}\right) \geq \frac{31}{36} ~\bigg|~ \exists t \in [0,s] : \rho_t = 0 \right) = 1-\delta,
	\end{align*}
	and so we obtain one half of our desired result: for any $x \in [0,\infty)$,
	\begin{equation}\label{eq: minkowski_upper_lower}
		\lim_{s \to \infty} \mathbb{P}\left( \textnormal{dim}^\textnormal{M}(\{t \in [0,s] : \rho_t \leq x\}) \geq \frac{31}{36}\right) = 1.
	\end{equation}
	
	To prove the upper bound, we assume that $ka_k \to \infty$ and apply a covering argument inspired by the proof of \cite[Thm.~6.3]{SS10}, which shows that the Hausdorff dimension of the set of exceptional times in dynamical percolation is at most $31/36$. We give the main estimate as a lemma. 
	\begin{Lemma}\label{lem: expectation_covering_upper}
		There exists $\Cl[lgc]{C: covering}>0$ such that the following holds. For any $\epsilon>0$, $x,s \in [0,\infty)$, and $n\geq 1$ such that
		\begin{equation}\label{eq: n_condition}
			2^n \leq  L\left( 1-e^{-\epsilon}\left(\frac{1}{2} - \epsilon\right)\right), 
		\end{equation}
		the expected covering number satisfies
		\[
		\mathbb{E} N(\{t \in [0,s] : \rho_t \leq x\}, \epsilon) \leq  \Cr{C: covering}^{y+1} \left\lceil \frac{s}{\epsilon} \right\rceil \binom{n}{y}  \pi_1(2^n),
		\]
		where
		\begin{equation}\label{eq: s_def}
			y = \min\left\{ \left\lfloor \frac{x}{F^{-1}\left(\frac{1}{2} + \epsilon\right)} \right\rfloor, n\right\}.
		\end{equation}
	\end{Lemma}
	\begin{proof}
		Pick $\epsilon,x,s,n$ satisfying the conditions in the statement of the lemma. By covering our set with intervals of length $\epsilon$, we see that 
		\begin{align*}
			&N(\{t \in [0,s] : \rho_t \leq x\}, \epsilon) \\
			\leq&N(\{t \in [0,s] : T_t(0,\partial B(2^n)) \leq x\}, \epsilon) \\
			\leq~& \#\left\{i = 1, \dots, \left\lceil \frac{s}{\epsilon}\right\rceil : \exists t \in [(i-1)\epsilon,i\epsilon] \text{ for which } T_t(0,\partial B(2^n)) \leq x\right\}.
		\end{align*}
		Taking expectation and using time stationarity of our process, we obtain
		\begin{equation}\label{eq: expectation_mink_bound}
			\mathbb{E}N(\{t \in [0,s] : \rho_t \leq x\}, \epsilon) \leq \left\lceil \frac{s}{\epsilon} \right\rceil \mathbb{P}\left(\exists t \in [0,\epsilon] \text{ for which } T_t (0,\partial B(2^n)) \leq x\right).
		\end{equation}
		
		To bound the probability on the right of \eqref{eq: expectation_mink_bound}, we introduce an auxiliary percolation process. For $v \in \mathbb{Z}^2$, define
		\[
		\sigma_v = \begin{cases}
			1 & \quad \text{if } \omega_v(0) \geq \frac{1}{2} +\epsilon \text{ and } \mathfrak{s}_v \text{ does not increment in } [0,\epsilon) \\
			0 &\quad \text{otherwise}.
		\end{cases}
		\]
		The vertices $v$ with $\sigma_v = 1$ are those whose weights are initially sufficiently large and also do not update until at least time $\epsilon$. The family $(\sigma_v)_{v \in \mathbb{Z}^2}$ is i.i.d. and satisfies
		\begin{equation*}
			\mathbb{P}(\sigma_v = 1) =  e^{-\epsilon} \left( \frac{1}{2} - \epsilon\right) = 1 - \mathbb{P}(\sigma_v = 0).
		\end{equation*}
		The important property of the variables $\sigma_v$ is that if there is a $t \in [0,\epsilon]$ such that $T_t (0,\partial B(2^n)) \leq x$ (from the right side of \eqref{eq: expectation_mink_bound}), then there must be a path $\gamma$ connecting $0$ and $\partial B(2^n)$ such that 
		\begin{equation}\label{eq: sigma_part}
			\#\{v \in \gamma : \sigma_v = 1, v \neq 0\} \leq \frac{x}{F^{-1}\left( \frac{1}{2} + \epsilon\right)}.
		\end{equation}
		Indeed, if there is such a $t$, then there is a path $\gamma$ connecting $0$ and $\partial B(2^n)$ such that $\sum_{v \in \gamma, v \neq 0} \tau_v(t) \leq x$. But because any $v$ with $\sigma_v = 1$ has $\tau_v(0) = \tau_v(t)$, we find 
		\[
		x \geq \sum_{v \in \gamma, v \neq 0} \tau_v(t) \sigma_v = \sum_{v \in \gamma, v \neq 0} \tau_v(0) \sigma_v \geq F^{-1}\left( \frac{1}{2} + \epsilon\right) \sum_{v \in \gamma, v \neq 0} \sigma_v.
		\]
		The $\sigma$ variables are $0/1$-valued, so this implies \eqref{eq: sigma_part}. In other words,
		\begin{equation}\label{eq: split_part}
			\mathbb{P}\left( \exists t \in [0,\epsilon] \text{ for which } T_t(0,\partial B(2^n)) \leq x\right) \leq \mathbb{P}\left( \exists \gamma : 0 \to \partial B(2^n) \text{ satisfying }\eqref{eq: sigma_part} \right).
		\end{equation}
		
		Now we use percolation tools to bound the right side of \eqref{eq: split_part}. Such a path $\gamma$ has at most $x/F^{-1}(1/2 + \epsilon)$ many nonzero vertices with $\sigma$-weight equal to 1, and each of these vertices must be in an annulus of the form $\text{Ann}(2^{j-1},2^j)$ for some integer $j$ satisfying $0 \leq j \leq n$. Because all other vertices on $\gamma$ have sigma-weight 0, we can therefore find a sequence $(k_j)_{j=0}^{y+1}$ with
		\[
		y \leq \frac{x}{F^{-1}\left( \frac{1}{2} + \epsilon\right)} \text{ and } 1 = k_0 \leq k_1 < k_2 < \dots < k_y \leq k_{y+1} = n
		\]
		such that for each $\ell= 0, \dots, y$, $B(2^{k_\ell})$ is connected to $\partial B(2^{k_{\ell+1}-1})$ by a path whose vertices $v$ have $\sigma_v=0$. (If $k_\ell = k_{\ell+1}$, as might be the case for $\ell=0$ or $\ell = y$, we use the convention throughout that the corresponding connection always exists.) In fact, this statement remains true if we assume that $y$ takes the value \eqref{eq: s_def}. So, by a union bound and independence, we find
		\begin{equation}\label{eq: split_part_2}
			\mathbb{P}\left( \exists \gamma : 0 \to \partial B(2^n) \text{ satisfying }\eqref{eq: sigma_part}\right)\leq \sum_{1 \leq k_1 < \cdots < k_y \leq n} \prod_{\ell=0}^{y} \pi_1\left( 1-e^{-\epsilon}\left( \frac{1}{2}-\epsilon\right); 2^{k_\ell}, 2^{k_{\ell+1}-1}\right).
		\end{equation}
		By \eqref{eq: near_critical_arm_event_comparability} and \eqref{eq: n_condition}, we can bound each factor on the right of \eqref{eq: split_part_2} by $\Cr{C: change_p}\pi_1(2^{k_\ell}, 2^{k_{\ell+1}-1})$. Because we can find $\Cl[smc]{c: annulus}>0$ such that for all $m \geq 0$, one has $\pi_1(2^m,2^{m+1}) \geq \Cr{c: annulus}$ (see \eqref{eq: near_critical_crossing}), the right side of \eqref{eq: split_part_2} is bounded as
		\begin{align}
			&\sum_{1 \leq k_1 < \cdots < k_y \leq n} \prod_{\ell=0}^{y} \pi_1\left( 1-e^{-\epsilon}\left( \frac{1}{2} - \epsilon\right); 2^{k_\ell}, 2^{k_{\ell+1}-1}\right) \nonumber \\
			\leq~& \left( \frac{\Cr{C: change_p}}{\Cr{c: annulus}}\right)^{y+1} \sum_{1 \leq k_1 < \cdots < k_y \leq n} \left[ \prod_{\ell=0}^{y} \pi_1\left( 2^{k_\ell}, 2^{k_{\ell+1}-1}\right) \pi_1 \left(2^{k_{\ell+1}-1}, 2^{k_{\ell+1}}\right)\right]. \label{eq: apply_QM}
		\end{align}
		We apply quasimultiplicativity from \eqref{eq: quasimultiplicativity} in the right side $2y+2$ many times to obtain
		\begin{align*}
			\sum_{1 \leq k_1 < \cdots < k_y \leq n} \prod_{\ell=0}^{y} \pi_1\left( 1-e^{-\epsilon}\left( \frac{1}{2} - \epsilon\right); 2^{k_\ell}, 2^{k_{\ell+1}-1}\right)  &\leq  \left( \frac{\Cr{C: change_p}}{\Cr{c: annulus} \Cr{c: c_quasi}^2}\right)^{y+1} \sum_{1 \leq k_1 < \cdots < k_y \leq n} \pi_1\left(2,2^n\right).\\
			&\leq \pi_1\left( 2,2^n\right) \binom{n}{y} \left( \frac{\Cr{C: change_p}}{\Cr{c: annulus} \Cr{c: c_quasi}^2}\right)^{y+1} \\
			&\leq \Cr{C: covering}^{y+1} \binom{n}{y} \pi_1(2^n).
		\end{align*}
		We place this bound in \eqref{eq: split_part_2}, and then go back through \eqref{eq: split_part} and \eqref{eq: expectation_mink_bound} to arrive at the statement of Lemma~\ref{lem: expectation_covering_upper}.
	\end{proof}

	Returning to the upper bound for the Minkowski dimension, we let $x$ and $s$ in $[0,\infty)$, with a fixed $\delta>0$, and we choose $\Cl[lgc]{C: x_inequality}$ so large that 
	\begin{equation}\label{eq: x_choice}
		x \leq \frac{\Cr{C: x_inequality}}{2} \left( \frac{4}{3} - \frac{\delta}{2} \right).
	\end{equation}
	Because we have assumed $ka_k \to \infty$, we have $ka_k \geq \Cr{C: x_inequality}$ for all large $k$. For such $k$, we make the choices
	\[
	\epsilon = 2^{-k},~n = \left\lceil \left( \frac{4}{3} - \frac{\delta}{2} \right)k \right\rceil
	\]
	for use in Lemma~\ref{lem: expectation_covering_upper}. One can check that \eqref{eq: n_condition} holds for large $k$ and the parameter $y$ in \eqref{eq: s_def} satisfies
	\[
	y \leq \min\left\{ \frac{x}{a_k}, \left\lceil k \left( \frac{4}{3} - \frac{\delta}{2}\right)\right\rceil \right\} \leq \min\left\{ \frac{x}{\Cr{C: x_inequality}}k, \left\lceil k \left( \frac{4}{3} - \frac{\delta}{2}\right)\right\rceil\right\} = \frac{x}{\Cr{C: x_inequality}}k,
	\]
	which is $\leq n/2$. This implies that if $k$ is large,
	\[
	\binom{n}{y} \leq \binom{n}{\left\lfloor \frac{x}{\Cr{C: x_inequality}} k\right\rfloor} \leq \left( \frac{e\left\lceil \left( \frac{4}{3} - \frac{\delta}{2}\right)k \right\rceil}{\left\lfloor \frac{x}{\Cr{C: x_inequality}} k \right\rfloor}\right)^{\left\lfloor \frac{x}{\Cr{C: x_inequality}} k \right\rfloor}\leq \left( \frac{4e \Cr{C: x_inequality} \left( \frac{4}{3} - \frac{\delta}{2} \right)}{x}\right)^{\frac{x}{\Cr{C: x_inequality}}k}.
	\]
	Plugging these into Lemma~\ref{lem: expectation_covering_upper}, we obtain for large $k$
	\[
	\mathbb{E} N(\{t \in [0,s] : \rho_t \leq x\}, 2^{-k}) \leq  \Cr{C: covering}^{\frac{x}{\Cr{C: x_inequality}}k+1} \left\lceil s2^k \right\rceil \left( \frac{4e \Cr{C: x_inequality} \left( \frac{4}{3} - \frac{\delta}{2} \right)}{x}\right)^{\frac{x}{\Cr{C: x_inequality}}k}  \pi_1\left(2^{\left( \frac{4}{3} - \frac{\delta}{2} \right) k}\right).
	\]
	We choose $C_6$ so large that this implies
	\begin{equation}\label{eq: almost_end}
		\mathbb{E} N(\{t \in [0,s] : \rho_t \leq x\}, 2^{-k}) \leq \left\lceil s2^k\right\rceil 2^{\frac{\delta}{2} k} \pi_1\left( 2^{\left( \frac{4}{3} - \frac{\delta}{2}\right)k}\right)
	\end{equation}
	whenever $k$ is large enough. Use the value $5/48$ of the one-arm exponent from \eqref{eq: arm_exponents_1_2_4} to obtain for large $k$
	\begin{equation}\label{eq: upper_end}
		\mathbb{E} N(\{t \in [0,s] : \rho_t \leq x\}, 2^{-k}) \leq \left\lceil s2^k\right\rceil 2^{-\left(\frac{5}{36}- \frac{3\delta}{4}\right)k}.
	\end{equation}
	
	Markov's inequality and the Borel-Cantelli lemma give a.s.
	\[
	N(\{t \in [0,s] : \rho_t \leq x\}, 2^{-k}) \leq \left\lceil s2^k\right\rceil 2^{-\left(\frac{5}{36}- \delta\right)k} \text{ for all large }k.
	\]
	This implies that a.s., $\textnormal{dim}^\textnormal{M}(\{t \in [0,s] : \rho_t \leq x\}) \leq 31/36+\delta$ (and by taking $\delta \downarrow 0$ through a countable set we get the upper bound $31/36$), and therefore
	\begin{equation}\label{eq: end_upper_bound}
		\lim_{s \to \infty} \mathbb{P}\left( \textnormal{dim}^\textnormal{M}(\{t \in [0,s] : \rho_t \leq x\}) \leq \frac{31}{36}\right) = 1.
	\end{equation}
	We combine this with \eqref{eq: minkowski_upper_lower} to complete the proof of Theorem~\ref{thm: Minkowski}(1).

	\subsubsection{Theorem~\ref{thm: Minkowski}(2)}

	The proof of Theorem~\ref{thm: Minkowski}(2) will be split into four steps. 
	
	\paragraph{Step 1:} Constructing many time intervals in which $T_t(0,\partial B(2^n)) = 0$.
	
	To give a lower bound on the set of exceptional times, we must show there are many times $t$ at which $\rho_t$ is small. First, we construct many times at which $T_t(0,\partial B(2^n))=0$, for some large $n$, and later we connect these boxes to infinity by low-weight paths. To glue connections together properly, we will need to ensure $T_t(0,\partial B(2^n))=0$ for all $t$ in a small interval.
	
	Define the event $A_t = A_t(n,M)$ for $t \in [0,1]$ and $n,M\geq 1$ by the existence of a circuit $\mathcal{C}$ around 0 in $\text{Ann}(2^n,2^{n+1})$ and a path $\gamma$ connecting $0$ to $\mathcal{C}$ in the interior of $\mathcal{C}$ such that 
	\begin{enumerate}
		\item $T_t(\gamma) = T_t(\mathcal{C}) = 0$, and 
		\item for each vertex $v \in \gamma \cup \mathcal{C}$, the Poisson process $\mathfrak{s}_v$ does not increment in the interval $[t,t+1/M)$.
	\end{enumerate}
	Also let
	\[
	\mathfrak{N}(n,M) = \#\left\{i = 0, \dots, M-1 : A_{\frac{i}{M}} \text{ occurs}\right\}.
	\]
	We will show that there is $\Cl[smc]{c: second_moment}>0$ such that for $n,M \geq 1$ satisfying
	\begin{equation}\label{eq: step_1_conditions}
		L\left( \frac{1}{2}e^{-\frac{1}{M}}\right) \geq 2^{n+1} \text{ and } M\pi_1(2^{n+1}) \geq 1,
	\end{equation}
	one has
	\begin{equation}\label{eq: step_1}
		\mathbb{P}\left( \mathfrak{N}(n,M) \geq \Cr{c: second_moment} M \pi_1(2^{n+1})\right) \geq \Cr{c: second_moment}.
	\end{equation}
	
	To show \eqref{eq: step_1}, we apply the second moment method, using another auxiliary percolation process. For $v \in \mathbb{Z}^2$, set
	\[
	\sigma_v = \begin{cases}
		1 &\quad\text{if } \tau_v(0)  > 0 \text{ or } \mathfrak{s}_v \text{ increments in } [ 0,\frac{1}{M}) \\
		0 &\quad\text{otherwise.}
	\end{cases}
	\]
	Then the $\sigma_v$'s are i.i.d.~and satisfy
	\[
	\mathbb{P}(\sigma_v=0) = \frac{1}{2}e^{-\frac{1}{M}} = 1 - \mathbb{P}(\sigma_v=1).
	\]
	Using the FKG inequality,
	\begin{align*}
		\mathbb{E}\mathfrak{N}(n,M) =  M\mathbb{P}(A_0) &\geq M\mathbb{P}(0 \to \partial B(2^{n+1}) \text{ by a path of vertices }v \text{ with } \sigma_v=0) \\
		&\times \mathbb{P}(\text{some circuit around 0 in Ann}(2^n,2^{n+1}) \text{ has all vertices }v \text{ with } \sigma_v=0)
	\end{align*}
	Equation \eqref{eq: near_critical_crossing} implies that, under the assumption $L(e^{-1/M}/2) \geq 2^{n+1}$, the first probability factor on the right is bounded below by $\Cl[smc]{c: hat_head}\pi_1(2^{n+1})$, and the second factor is bounded below by $\Cr{c: hat_head}$. Therefore
	\begin{equation}\label{eq: first_moment}
		\mathbb{E}\mathfrak{N}(n,M) \geq \Cr{c: hat_head}^2 M\pi_1(2^{n+1}) \text{ if } L\left( \frac{1}{2} e^{-\frac{1}{M}}\right) \geq 2^{n+1}.
	\end{equation}
	
	To estimate the second moment of $\mathfrak{N}(n,M)$, we write
	\[
	\mathbb{E}\mathfrak{N}(n,M)^2 = \sum_{i,j=0}^{M-1} \mathbb{P}\left(A_{\frac{i}{M}} \cap A_{\frac{j}{M}}\right) \leq 2M\sum_{i=0}^{M-1} \mathbb{P}\left(A_0 \cap A_{\frac{i}{M}}\right) \leq 2M \sum_{i=0}^{M-1} \mathbb{P}\left( B_0 \cap B_{\frac{i}{M}}\right),
	\]
	where $B_t$ is the event that $T_t(0,\partial B(2^n)) = 0$. By \cite[Eq.~(5.4)]{SS10}, there is $\Cl[lgc]{C: noise}>0$ such that 
	\[
	\mathbb{P}(B_0 \cap B_t) \leq \Cr{C: noise} \pi_1(2^n)^2 t^{-\frac{7}{8}} \text{ for } t \in [0,1].
	\]
	This gives
	\begin{align*}
		\mathbb{E}\mathfrak{N}(n,M)^2 &\leq 2M \pi_1(2^n) + 2\Cr{C: noise}M^2 \pi_1(2^n)^2 \cdot \frac{1}{M} \sum_{i=1}^{M-1} \left( \frac{i}{M}\right)^{-\frac{7}{8}} \\
		&\leq 2M\pi_1(2^n) + 2\Cr{C: noise}M^2 \pi_1(2^n)^2 \int_0^1 t^{-\frac{7}{8}}~\text{d}t \\
		&\leq \C[lgc] \left[ M\pi_1(2^n) + M^2 \pi_1(2^n)^2\right].
	\end{align*}
	Quasimultiplicativity of $\pi_1$ from \eqref{eq: quasimultiplicativity} along with \eqref{eq: near_critical_crossing} gives $\pi_1(2^n) \leq \C[lgc] \pi_1(2^{n+1})$. Because $M\pi_1(2^{n+1}) \geq 1$, we obtain the upper bound
	\[
	\mathbb{E}\mathfrak{N}(n,M)^2 \leq \Cl[lgc]{C: next_line} M^2 \pi_1(2^{n+1})^2,
	\]
	which by \eqref{eq: first_moment} is $\leq \left( \Cr{c: hat_head}^4/\Cr{C: next_line}\right) (\mathbb{E}\mathfrak{N}(n,M))^2$. The Paley-Zygmund inequality finishes the proof of \eqref{eq: step_1}.
 
		\begin{figure}
			\centering
			\includegraphics[width=0.6\linewidth]{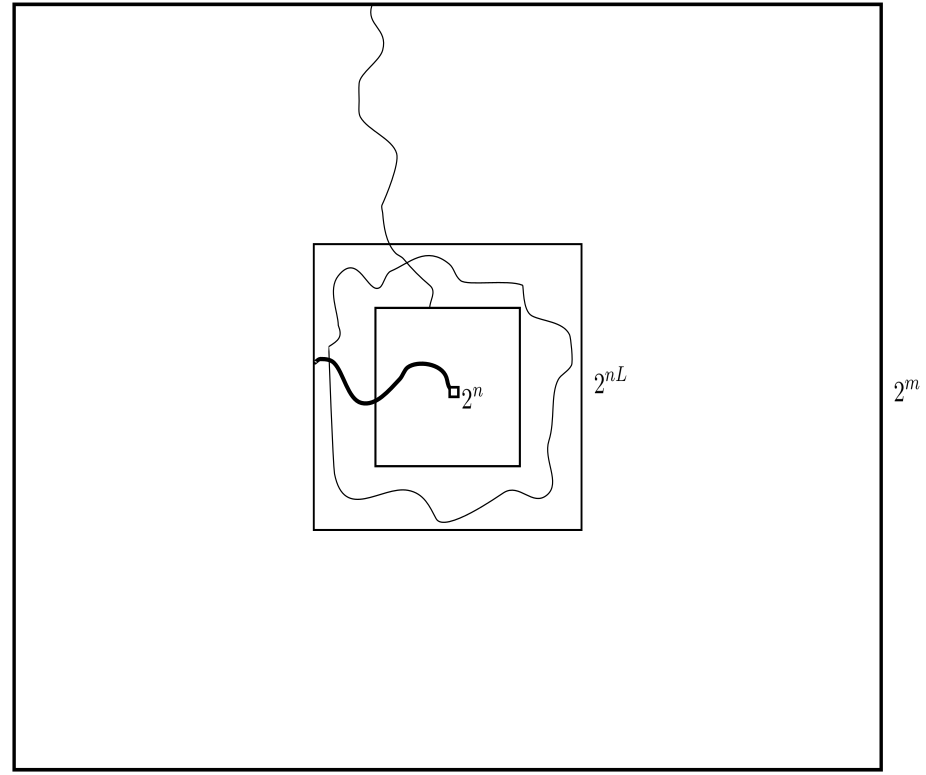}
			\caption{Depiction of the event $F_m$ in Step 2 of the proof of Theorem \ref{thm: Minkowski}(2). The circuit in the annulus $B(2^{nL}) \backslash B(2^{nL-1})$ surrounding the origin has zero passage time and corresponds to item 1. The (boldly drawn) path from $B(2^n)$ to $\partial B(2^{nL})$ has low passage time, corresponding to item 2. The path which reaches $\partial B(2^m)$ has zero passage time, corresponding to item 3. }
			\label{fig: Minkowski}
		\end{figure}

	\paragraph{Step 2:} Constructing times at which $T_t(B(2^n),\infty)$ is small.
	
	In this step, we prove that intervals that are not too small have positive probability to contain a time $t$ at which $B(2^n)$ is connected to infinity by a low-weight path. Precisely, we will show that there exist $\Cl[lgc]{C: step_2},\Cl[smc]{c: step_2}, \delta>0$ such that for all large $n$ and all $L\geq 1$,
	\begin{equation}\label{eq: step_2}
		\mathbb{P}\left( T_t\left( B(2^n), \infty\right) \leq \Cr{C: step_2} nL a_{\lfloor \delta n \rfloor}\text{ for some } t \in \bigg[ 0, 2^{-\frac{nL}{9}}\bigg)\right) \geq \Cr{c: step_2}.
	\end{equation}
	Here, $T_t(B(2^n), \infty) = \lim_{m \to \infty} T_t(B(2^n), \partial B(2^m))$. Because $\rho_t=\infty$ a.s.~for any fixed $t$, it will suffice to prove \eqref{eq: step_2} with $[0,2^{-nL/9})$ replaced by the closed interval $[0,2^{-nL/9}]$.
	
	To prove \eqref{eq: step_2} we again use the second moment method and begin by approximating the event in question by one that depends on the state of finitely many vertices. For $m \geq nL$, let $F_m = F_m(n,L,\delta,\Cr{C: step_2})$ be the event defined by the following conditions (see Figure \ref{fig: Minkowski}):
	\begin{enumerate}
		\item there is a circuit around 0 in $\text{Ann}(2^{nL-1},2^{nL})$ whose vertices $v$ have $\tau_v=0$,
		\item $T(B(2^n),\partial B(2^{nL})) \leq \Cr{C: step_2} nL a_{\lfloor \delta n\rfloor}$, and
		\item $T\left( B(2^{nL-1}),\partial B(2^m)\right) = 0$.
	\end{enumerate}
	Let $F_m^t$ be the event that $F_m$ occurs in the configuration $(\tau_v(t))$. Last, define
	\[
	Y = \int_0^{2^{-\frac{nL}{9}}} \mathbf{1}_{F_m^t}~\text{d}t.
	\]
	The first moment of $Y$ is estimated as
	\[
	\mathbb{E}Y = \int_0^{2^{-\frac{nL}{9}}} \mathbb{P}(F_m^t)~\text{d}t = 2^{-\frac{nL}{9}} \mathbb{P}(F_m),
	\]
	and by the FKG inequality we obtain
	\begin{equation}\label{eq: FKG_breakdown}
		\mathbb{E}Y \geq 2^{-\frac{nL}{9}} \mathbb{P}(\text{item 1})\mathbb{P}(\text{item 2})\mathbb{P}(\text{item 3}).
	\end{equation}
	\eqref{eq: near_critical_crossing} gives a $\Cl[smc]{c: throw_away}$ such that $\mathbb{P}(\text{item 1}) \geq \Cr{c: throw_away}$ and quasimultiplicativity gives $\mathbb{P}(\text{item 3}) \geq \Cr{c: throw_away} \pi_1(2^{nL},2^m)$. 
	
	For item 2, we claim that for some $\Cr{C: step_2},\delta>0$,
	\begin{equation}\label{eq: momentless_bound}
		\mathbb{P}(\text{item 2}) \geq \frac{1}{2} \text{ for } n \text{ large and all } L \geq 1.
	\end{equation}
	We follow much of what was laid out in Section~\ref{sec: ruling_out}. First recall the definition of $\mathsf{T}(n)$ from \eqref{eq: T(n)_def} and observe that for large $n$,
	\[
	\mathbb{P}(\text{item 2 fails}) \leq \mathbb{P}\left( \sum_{k=n}^{nL} \mathsf{T}(k) \geq \Cr{C: step_2} nLa_{\lfloor \delta n \rfloor}\right) \leq \mathbb{P}\left( \sum_{k=n}^{nL} \frac{\mathsf{T}(k)}{F^{-1}(p_{\lceil 1.5^k\rceil})} \geq \Cr{C: step_2} nL\right),
	\]
	so long as $\delta$ is small. By bounding the right side, we obtain
	\begin{equation}\label{eq: nachos_testaverde}
		\mathbb{P}(\text{item 2 fails}) \leq \mathbb{P}\left( \sum_{k=n}^{nL} Y_k \geq \frac{\Cr{C: step_2}}{\Cr{c: stretched}} nL\right) + \mathbb{P}(\exists k \geq n : X_k \neq Y_k),
	\end{equation}
	where
	\[
	X_k = \frac{\mathsf{T}(k)}{\Cr{c: stretched} F^{-1}\left( p_{\lceil 1.5^k\rceil}\right)} \text{ and } Y_k = X_k \mathbf{1}_{\{X_k \leq (4/3)^{\eta k}\}},
	\]
	with $\eta = 2 - \alpha/4$ (this choice corresponds to $\epsilon=\alpha/2$ in \eqref{eq: eta_def}). Just as in \eqref{eq: Y_bound}, we have
	\[
	\mathbb{P}(Y_k \geq u) \leq \Cr{C: pastaroni}\exp\left(-\Cr{c: stretched}u^{\frac{1}{\eta}}\right) \text{ for } u \geq 0 \text{ and } k \geq 2,
	\]
	and so the sequence $(\mathbb{E}Y_k)$ is bounded. Therefore we can continue from \eqref{eq: nachos_testaverde} and choose $\Cr{C: step_2}>0$ such that 
	\begin{equation}\label{eq: last_eq_pizza_head}
		\mathbb{P}(\text{item 2 fails}) \leq \frac{1}{4} + \mathbb{P}(\exists k \geq n : X_k \neq Y_k).
	\end{equation}
	Like in \eqref{eq: X_n_Y_n}, we have $\mathbb{P}(X_k \neq Y_k) \leq 18^{-(k+2)}$ for large $k$, so a union bound on the right side of \eqref{eq: last_eq_pizza_head} shows \eqref{eq: momentless_bound}.
	
	
	Putting \eqref{eq: momentless_bound} and the bounds on items 1 and 3 into \eqref{eq: FKG_breakdown} produces for our chosen $\delta$ and $\Cr{C: step_2}$,
	\begin{equation}\label{eq: first_moment_Y}
		\mathbb{E}Y \geq \Cl[smc]{c: lagatta}2^{-\frac{nL}{9}} \pi_1\left( 2^{nL}, 2^m\right) \text{ for large }n \text{ and all } L \geq 1, m \geq nL.
	\end{equation}
	
	The second moment is
	\begin{align}
		\mathbb{E}Y^2 &= \int_0^{2^{-\frac{nL}{9}}} \int_0^{2^{-\frac{nL}{9}}} \mathbb{P}\left( F_m^t\cap F_m^s\right)~\text{d}s~\text{d}t \nonumber\\
		&\leq 2 \cdot 2^{-\frac{nL}{9}} \int_0^{2^{-\frac{nL}{9}}} \mathbb{P}\left( F_m^0 \cap F_m^t\right)~\text{d}t \nonumber \\
		&\leq 2\cdot 2^{-\frac{nL}{9}} \int_0^{2^{-\frac{nL}{9}}} \mathbb{P}\left( W_0(2^{nL},2^m) \cap W_t(2^{nL},2^m)\right)~\text{d}t. \label{eq: Y_second_moment_1}
	\end{align}
	Here, $W_t(r,R)$ is the event that $T_t\left( B(r),\partial B(R)\right) = 0$. To estimate the probability in the integral we give a lemma that closely follows \cite[p.~647]{SS10}.
	\begin{Lemma}
		There exists $\Cl[lgc]{C: noise_2}>0$ such that for all large $r$, if $R \geq r$ and $t \leq r^{-1/9}$, then
		\[
		\mathbb{P}\left( W_0(r,R) \cap W_t(r,R)\right) \leq \Cr{C: noise_2} \pi_1(r,R)^2 \pi_1\left( r, \left\lceil t^{-9}\right\rceil\right)^{-1}.
		\]
	\end{Lemma}
	
	\begin{proof}
		The proof of this lemma uses discrete Fourier analysis and is similar to the corresponding arguments of \cite{SS10}, so we only sketch the idea. With the notation $f_r^R = \mathbf{1}_{W(r,R)}$, where $W(r,R)$ is the event that $T(B(r), \partial B(R)) = 0$, \cite[Cor.~4.5]{SS10} states that
		\[
		\sum_{\#S=k} \hat{f}_r^R(S)^2 \leq kr^{o(1)}\pi_1(r,R)^2 \pi_2(r) \text{ for } 1 \leq r \leq R \text{ and } k \geq 1,
		\]
		where $\hat{f}_r^R(S)$ is the Fourier-Walsh coefficient associated to a subset $S$ of the vertices of $B(R)$, and the sum is over all such sets of cardinality $k$. The $r^{o(1)}$ factor depends on $r$ only and not on $R$, and $o(1)$ is a term that converges to 0 as $r \to \infty$. 
		Exactly as in \cite[p.~647]{SS10}, we also have for any $s \in [r,R]$
		\[
		\mathbb{P}(W_0(r,R)\cap W_t(r,R)) \leq \pi_1(r,s) \mathbb{P}(W_0(s,R) \cap W_t(s,R)) = \pi_1(r,s) \sum_S e^{-t\#S} \hat{f}_s^R(S)^2.
		\]
		Combining these inequalities and using the fact that $\hat{f}_s^R(\emptyset) = \pi_1(s,R)$ produces the bound
		\[
		\mathbb{P}(W_0(r,R) \cap W_t(r,R)) \leq \pi_1(r,s) \left( \pi_1(s,R)^2 + s^{o(1)} \sum_{k=1}^\infty e^{-kt} k \pi_1(s,R)^2 \pi_2(s)\right).
		\]
		Because $\sum_{k=1}^\infty ke^{-kt} \leq C/t^2$, we obtain for $s \in [r,R]$
		\[
		\mathbb{P}(W_0(r,R) \cap W_t(r,R)) \leq \pi_1(r,s) \pi_1(s,R)^2 \left[ 1 + \frac{s^{o(1)}\pi_2(s)}{t^2}\right].
		\]
		Applying quasimultiplicativity, we obtain for $s \in [r,R]$
		\[
		\mathbb{P}(W_0(r,R) \cap W_t(r,R)) \leq \Cr{c: c_quasi}^{-1} \pi_1(r,R)^2 \pi_1(r,s)^{-1} \left[ 1+ \frac{s^{o(1)} \pi_2(s)}{t^2}\right].
		\]
		Although we had the restriction $s \in [r,R]$, this inequality remains true for $s \geq R$. We can therefore take $s = \lceil t^{-9}\rceil$, which is $\geq r$ by assumption, to obtain
		\begin{equation}\label{eq: SS_almost_done}
			\mathbb{P}(W_0(r,R) \cap W_t(r,R)) \leq \Cr{c: c_quasi}^{-1} \pi_1(r,\lceil t^{-9}\rceil)^{-1} \pi_1(r,R)^2 \left[ 1+ \frac{t^{o(1)} \pi_2(\lceil t^{-9}\rceil)}{t^2}\right].
		\end{equation}
		The exact value $1/4$ of the two-arm exponent from \eqref{eq: arm_exponents_1_2_4} implies $\pi_2(\lceil t^{-9}\rceil) = t^{9/4+o(1)}$ as $t \downarrow 0$. Placing this in \eqref{eq: SS_almost_done} completes the proof of the lemma.
	\end{proof}
	
	Returning to \eqref{eq: Y_second_moment_1}, we use the lemma under the assumption that $n$ is large, $L \geq 1$, and $m \geq nL$ to find
	\begin{equation}\label{eq: Y_second_moment_2}
		\int_0^{2^{-\frac{nL}{9}}} \mathbb{P}\left( W_0(2^{nL},2^m) \cap W_t(2^{nL},2^m)\right)~\text{d}t \leq \Cr{C: noise_2} \pi_1(2^{nL}, 2^m)^2 \int_0^{2^{-\frac{nL}{9}}}  \pi_1\left( 2^{nL}, \left\lceil t^{-9}\right\rceil\right)^{-1} ~\text{d}t.
	\end{equation}
	The exact value of the one-arm exponent in \eqref{eq: arm_exponents_1_2_4} provides, for any $\alpha > 5/48$, a constant $\Cl[smc]{c: one_arm}$ such that $\pi_1\left( 2^{nL}, \left\lceil t^{-9}\right\rceil\right) \geq \Cr{c: one_arm}\left( 2^{nL} t^9 \right)^{\alpha}$ and so
	\[
	\int_0^{2^{-\frac{nL}{9}}}  \pi_1\left( 2^{nL}, \left\lceil t^{-9}\right\rceil\right)^{-1} ~\text{d}t \leq \Cr{c: one_arm}^{-1} 2^{-nL\alpha} \int_0^{2^{-\frac{nL}{9}}} t^{-9\alpha}~\text{d}t = \Cr{c: one_arm}^{-1} 2^{-nL\alpha} \frac{2^{-\frac{nL}{9}(1-9\alpha)}}{1-9\alpha}
	\]
	assuming $\alpha \in (5/48,1/9)$. Putting this estimate back in \eqref{eq: Y_second_moment_2}, and then back in \eqref{eq: Y_second_moment_1} produces, upon comparison to \eqref{eq: first_moment_Y}, the upper bound
	\begin{equation}\label{eq: first_second_moment_comparison}
		\mathbb{E}Y^2 \leq \Cl[lgc]{C: I_am_new} \cdot \left( 2^{-\frac{nL}{9}} \pi_1(2^{nL},2^m)\right)^2 \leq \frac{\Cr{C: I_am_new}}{\Cr{c: lagatta}^2} \left(\mathbb{E}Y\right)^2 \text{ for large }n \text{ and all } L \geq 1, m\geq nL.
	\end{equation}
	
	Now that we have \eqref{eq: first_second_moment_comparison}, we can apply the Paley-Zygmund inequality to find $\Cl[smc]{c: PZ}>0$ such that 
	\[
	\mathbb{P}\left( Y \geq \Cr{c: PZ} 2^{-\frac{nL}{9}} \pi_1\left( 2^{nL},2^m\right)\right) \geq \Cr{c: PZ} \text{ for large }n \text{ and all } L \geq 1, m\geq nL.
	\]
	This implies
	\[
	\mathbb{P}\left( F_m^t \text{ occurs for some } t \in \left[ 0, 2^{-\frac{nL}{9}}\right]\right) \geq \Cr{c: PZ} \text{ for large }n \text{ and all } L \geq 1, m \geq nL
	\]
	and so
	\begin{equation}\label{eq: io}
		\mathbb{P}\left( \cup_{t \in \left[ 0, 2^{-\frac{nL}{9}}\right]} F_m^t \neq \emptyset \text{ for infinitely many } m\right) \geq \Cr{c: PZ} \text{ for large }n \text{ and all } L \geq 1.
	\end{equation}
	
	To complete the proof of \eqref{eq: step_2} and therefore to move to step 3, we must show that a.s., the event $\{\cup_t F_m^t \neq \emptyset \text{ for infinitely many }m\}$ in \eqref{eq: io} implies the event in the probability in \eqref{eq: step_2}. To do this, define the following sets of times corresponding to items 1-3 of the definition of $F_m$:
	\begin{enumerate}
		\item $F(1)$, the set of $t \in \left[ 0,2^{-\frac{nL}{9}}\right]$ such that 
		\begin{enumerate}
			\item there is a circuit around 0 in $\text{Ann}(2^{nL-1},2^{nL})$ whose vertices $v$ have $\tau_v(t) = 0$ and
			\item $T_t(B(2^n),\partial B(2^{nL})) \leq \Cr{C: step_2}nLa_{\lfloor \delta n\rfloor}$.
		\end{enumerate}
		\item $F_m(2)$ is the set of $t \in \left[ 0,2^{-\frac{nL}{9}}\right]$ such that $T_t\left( B(2^{nL-1}), \partial B(2^m)\right) = 0$.
	\end{enumerate}
	We also define a new process $(\bar{\tau}_v(t))$ by setting, for every vertex $v$, the set $\{t : \bar{\tau}_v(t) = 0\}$ to be the closure of the set $\{t : \tau_v(t) = 0\}$, and otherwise $\tau_v(t) = \bar{\tau}_v(t)$. Then \cite[Lem.~3.2]{HPS97} states that a.s., for every vertex $v$,
	\begin{equation}\label{eq: HPS_lemma}
		\begin{split}
			&\{t \geq 0 : v \text{ is in an infinite path of vertices }w \text{ with } \bar{\tau}_w(t) = 0\} \\
			=~& \{t \geq 0 : v \text{ is in an infinite path of vertices }w \text{ with } \tau_w(t) = 0\}.
		\end{split}
	\end{equation}
	Last, we need a consequence of the argument of \cite[Lem.~3.2]{HPS97}: a.s., there are no infinite clusters of zero-weight vertices at any times at which any vertex's Poisson process increments. That is,
	\begin{equation}\label{eq: HPS_exceptional_lemma}
		\mathbb{P}\left( \exists t \geq 0 : \begin{array}{c}
			\text{ there is an infinite path whose vertices }v\text{ have } \tau_v(t) = 0 \\
			\text{ and } \mathfrak{s}_w(t-) < \mathfrak{s}_w(t) \text{ for some }w
		\end{array}
		\right) = 0.
	\end{equation}
	
	Now fix an outcome in the intersection of the event in \eqref{eq: io}, the event that \eqref{eq: HPS_lemma} occurs for all $v$, and the complement of the event in the probability in \eqref{eq: HPS_exceptional_lemma}. We will argue that this outcome is in the event in the probability in \eqref{eq: step_2}. Because the event in \eqref{eq: io} occurs, we can find an increasing sequence $(m_k)$ of integers such that $F_{m_k}(2) \cap F(1)$ is nonempty. Let $S_{m_k}$ be the closure of $F_{m_k}(2) \cap F(1)$ and note that the $S_{m_k}$'s are compact and nested, so there exists a time $t \in [0,2^{-nL/9}]$ in their intersection. We will show that this $t$ satisfies $T_t(B(2^n),\infty) \leq \Cr{C: step_2}nLa_{\lfloor \delta n\rfloor}$. Due to the bound \eqref{eq: io}, this will show \eqref{eq: step_2}. By planarity, it will suffice to prove that $t \in F(1)$, and that $T_t(B(2^{nL-1}),\infty)=0$.
	
	To do this, we note that the closure of $F_{m_k}(2)$ equals the set defined by the same conditions as those in the definition of $F_{m_k}(2)$, but with $\tau$ replaced by $\bar{\tau}$. As $t$ is in this closure for all $k$, there is an infinite self-avoiding path intersecting $B(2^{nL-1})$ whose vertices $w$ have $\bar{\tau}_w(t) = 0$. This means that $t$ is in the (equal) sets in \eqref{eq: HPS_lemma} for some $v \in B(2^{nL-1})$, and therefore there is an infinite self-avoiding path whose vertices $w$ have $\tau_w(t) = 0$. Aiming for a contradiction, if we assume $t$ is not in $F(1)$, then because it is in the closure of $F(1)$, a set of disjoint half-open intervals (except possibly the point $2^{-nL/9}$), there must be some $w$ such that $\mathfrak{s}_w(t-)<\mathfrak{s}_w(t)$. Our outcome is in the complement of the event in the probability in \eqref{eq: HPS_exceptional_lemma}, so we obtain a contradiction. Therefore $t \in F(1)$ as well, and this completes step 2.

	\paragraph{Step 3.} Gluing.
	
	In this step, we glue together the events from the previous two steps to produce many times $t$ for which $\rho_t$ is small. Specifically, we find $\Cl[smc]{c: another_new}>0$ such that for $\Cr{C: step_2},\delta>0$ from \eqref{eq: step_2}, if $n$ is large, and $M$ satisfies \eqref{eq: step_1_conditions}, then for all $L \geq 1$, the covering number $N$ satisfies
	\begin{equation}\label{eq: step_3}
		\mathbb{P}\left( N\left( \left\{t \in [0,1] : \rho_t \leq \Cr{C: step_2}nLa_{\lfloor \delta n\rfloor} \right\}, 2^{-\frac{nL}{9}}\right) \geq \Cr{c: another_new} M\pi_1(2^{n+1}) \left\lfloor \frac{2^{\frac{nL}{9}}}{M}\right\rfloor\right) \geq \Cr{c: another_new}.
	\end{equation}
	
	We begin with the event from step 1. Because $M,n$ satisfy \eqref{eq: step_1_conditions}, if we define $S = S((x_i)) = \sum_i x_i$ for $x_0, \dots, x_{M-1} \in \{0,1\}$, then
	\begin{equation}\label{eq: step_3_indicators}
		\Cr{c: second_moment} \leq \mathbb{P}\left( \mathfrak{N}(n,M) \geq \Cr{c: second_moment} M \pi_1(2^{n+1})\right) = \sum_{(x_i) : S \geq \Cr{c: second_moment} M\pi_1(2^{n+1})} \mathbb{P}\left( \mathbf{1}_{A_{\frac{i}{M}}} = x_i \text{ for all } i\right).
	\end{equation}
	For any $i$ such that $A_{i/M}$ occurs, we will need to decouple the configurations inside and outside of the circuit $\mathcal{C}$. To do this, we define, for any circuit $\mathcal{C}$ around the origin in $\text{Ann}(2^n,2^{n+1})$, the event $\mathcal{D}_{i/M}(\mathcal{C})$ using the conditions
	\begin{enumerate}
		\item $\mathcal{C}$ is the innermost circuit in $\text{Ann}(2^n,2^{n+1})$ such that for all $v \in \mathcal{C}$, $\tau_v\left( \frac{i}{M}\right) = 0$ and $\mathfrak{s}_v$ does not increment in $[\frac{i}{M}, \frac{i+1}{M})$, and
		\item there exists a path $\gamma$ from $0$ to $\mathcal{C}$ in the interior of $\mathcal{C}$ such that for all $v \in \gamma$, $\tau_v\left( \frac{i}{M}\right) = 0$ and $\mathfrak{s}_v$ does not increment in $[\frac{i}{M}, \frac{i+1}{M})$.
	\end{enumerate}
	These conditions simply state that $\mathcal{C}$ is the innermost 0-circuit and $0$ is connected to $\mathcal{C}$ by a $0$-path, all in the weights $(\sigma_v)$ defined below \eqref{eq: step_1}. Therefore for distinct $\mathcal{C}, \mathcal{C}'$, the events $\mathcal{D}_{i/M} (\mathcal{C})$ and $\mathcal{D}_{i/M}(\mathcal{C}')$ are disjoint, and $A_{i/M} = \cup_{\mathcal{C}} \mathcal{D}_{i/M}(\mathcal{C})$. So we can decompose \eqref{eq: step_3_indicators} to obtain
	\begin{equation}\label{eq: step_3_circuit_decomposition}
		\Cr{c: second_moment} \leq \sum_{(x_i) : S \geq \Cr{c: second_moment} M\pi_1(2^{n+1})} \sum_{(\mathcal{C}_i)} \mathbb{P}\left( \mathbf{1}_{A_{\frac{i}{M}}} = 0 \text{ if } x_i = 0, \mathcal{D}_{\frac{i}{M}}(\mathcal{C}_i) \text{ occurs if } x_i=1\right).
	\end{equation}
	The inner ($(x_i)$-dependent) sum is understood to be over all choices of circuits $\mathcal{C}_i$ for those $i$ such that $x_i=1$.
	
	On the event in the sum of \eqref{eq: step_3_circuit_decomposition}, we will create many times at which the circuits $\mathcal{C}_i$ are connected to infinity by low-weight paths. So for $i$ such that $x_i=1$, we split the interval $[i/M,(i+1)/M)$ into contiguous subintervals
	\[
	\bigg[\frac{i}{M}, \frac{i}{M} + 2^{-\frac{nL}{9}}\bigg), \bigg[ \frac{i}{M} + 2^{-\frac{nL}{9}}, \frac{i}{M} + 2 \cdot 2^{-\frac{nL}{9}}\bigg), \dots,
	\]
	so that the total number of subintervals obtained is $\lfloor 2^{nL/9}/M\rfloor$. The collection $\mathcal{I} = \mathcal{I}((x_i))$ of all such subintervals, for all $i$ with $x_i = 1$, satisfies
	\begin{equation}\label{eq: I_cardinality}
		\#\mathcal{I} = S \left\lfloor \frac{2^{\frac{nL}{9}}}{M}\right\rfloor.
	\end{equation}
	Now for any $I \in \mathcal{I}$, let $i(I)$ be the value of $i$ such that $I \subset [i/M,(i+1)/M)$ and for a circuit $\mathcal{C}_i$ in $\text{Ann}(2^n,2^{n+1})$ around 0, let $B_I(\mathcal{C}_i)$ be the event that there exists $t \in I$ such that $T_t(\mathcal{C}_i,\infty) \leq \Cr{C: step_2}nL a_{\lfloor \delta n\rfloor}$. Because we have assumed that $n$ is large, we can apply inequality \eqref{eq: step_2} from step 2 and the Paley-Zygmund inequality to obtain
	\begin{equation}\label{eq: B_I_inequality}
		\mathbb{P}\left( B_I \text{ occurs for at least } \Cl[smc]{c: swedish_stunner} \#\mathcal{I} \text{ many } I \in \mathcal{I}\right) \geq \Cr{c: swedish_stunner}
	\end{equation}
	for some $\Cr{c: swedish_stunner}>0$. Indeed, the expected number of such $I$ is at least $\Cr{c: step_2}\#\mathcal{I}$ by \eqref{eq: step_2}, but the second moment is at most $(\#\mathcal{I})^2$.
	
	Next we observe that the event in the probability in \eqref{eq: step_3_circuit_decomposition} depends on $\tau_v(t)$ for (a) $v \in \mathcal{C}_i$ or in the interior of $\mathcal{C}_i$ and $t \in [i/M,(i+1)/M)$ for all $i$ such that $x_i=1$ and (b) $v \in B(2^{n+1})$ and $t \in [i/M,(i+1)/M)$ for all $i$ such that $x_i=0$. On the other hand, the event in the probability in \eqref{eq: B_I_inequality} depends on $\tau_v(t)$ for $v$ in the exterior of $\mathcal{C}_i$ and $t \in [i/M,(i+1)/M)$ for all $i$ such that $x_i=1$. Therefore these events are independent, and we write 
	\[
	\Cr{c: second_moment}\Cr{c: swedish_stunner} \leq \sum_{(x_i) : S \geq \Cr{c: second_moment} M\pi_1(2^{n+1})} \sum_{(\mathcal{C}_i)} \mathbb{P}\left( \begin{array}{c}
		\mathbf{1}_{A_{\frac{i}{M}}} = 0 \text{ if } x_i = 0, \mathcal{D}_{\frac{i}{M}}(\mathcal{C}_i) \text{ occurs if } x_i=1, \\
		B_I \text{ occurs for at least } \Cr{c: swedish_stunner} \#\mathcal{I} \text{ many } I \in \mathcal{I}
	\end{array}\right).
	\]
	For any $I \in \mathcal{I}$ such that $B_I$ occurs and $\mathcal{D}_{i/M}(\mathcal{C}_i)$ also occurs, there is a $t \in I$ such that $\rho_t \leq \Cr{C: step_2}nLa_{\lfloor \delta n \rfloor}$. Indeed, $\mathcal{D}_{i/M}(\mathcal{C}_i)$ occurs, so for all $s \in [i/M,(i+1)/M)$, $T_s(0,\mathcal{C}_i) = 0$ and $T_s(\mathcal{C}_i) = 0$, and furthermore for some $t \in I$ we also have $T_t(\mathcal{C}_i,\infty) \leq \Cr{C: step_2}nLa_{\lfloor \delta n\rfloor}$. Therefore after summing over $(\mathcal{C}_i)$, we obtain
	\[
	\Cr{c: second_moment}\Cr{c: swedish_stunner} \leq \sum_{(x_i) : S \geq \Cr{c: second_moment} M\pi_1(2^{n+1})} \mathbb{P}\left( \begin{array}{c}
		\mathbf{1}_{A_{\frac{i}{M}}} = x_i \text{ for all } i, \text{ at least }\Cr{c: swedish_stunner} \#\mathcal{I} \text{ many } I \in \mathcal{I} \\
		\text{ contain a }t \text{ such that } \rho_t \leq \Cr{C: step_2}nLa_{\lfloor \delta n \rfloor}
	\end{array}\right).
	\]
	Using \eqref{eq: I_cardinality} and summing over $(x_i)$ produces
	\[
	\Cr{c: second_moment}\Cr{c: swedish_stunner} \leq \mathbb{P}\left( \begin{array}{c}
		\text{at least }\Cr{c: swedish_stunner}\Cr{c: second_moment} M\pi_1(2^{n+1}) \left\lfloor \frac{2^{\frac{nL}{9}}}{M}\right\rfloor \text{ many disjoint intervals } \bigg[a, a+2^{-\frac{nL}{9}}\bigg) \\
		\text{in } [0,1]\text{ contain a }t \text{ such that } \rho_t \leq \Cr{C: step_2} nLa_{\lfloor \delta n \rfloor}
	\end{array}\right).
	\]
	This implies \eqref{eq: step_3} and completes step 3.

	\paragraph{Step 4.} Last, we use step 3 to lower bound the Minkowski dimension. Our assumption is $\liminf_{n \to \infty} na_n = 0$, so we can find a subsequence $(n_k)$ such that $n_ka_{\lfloor \delta n_k\rfloor} \to 0$, where $\delta$ is from \eqref{eq: step_2}. For $k \geq 1$, set $M = M(n_k) = \lceil 2^{\alpha n_k}\rceil$ for some $\alpha > 3/4$, and fix $L \geq 1$. Using the asymptotics for $L(p)$ in \eqref{eq: correlation_scaling} and the exact value $5/48$ of the one arm exponent from \eqref{eq: arm_exponents_1_2_4}, a direct computation shows that \eqref{eq: step_1_conditions} holds for $n_k$ and $M$, so we can apply inequality \eqref{eq: step_3} of step 3. For any $\beta > 5/48$, we obtain with probability $\geq \Cr{c: another_new}$
	\[
	N\left( \left\{t \in [0,1] : \rho_t \leq \Cr{C: step_2} n_kLa_{\lfloor \delta n_k \rfloor}\right\}, 2^{-\frac{n_kL}{9}}\right) \geq \Cr{c: another_new} \lceil 2^{\alpha n_k}\rceil 2^{-\beta n_k} \left\lfloor \frac{2^{\frac{n_kL}{9}}}{\left\lceil 2^{\alpha n_k}\right\rceil}\right\rfloor
	\]
	holds for infinitely many $k$. So long as $L$ is chosen to be $>9\alpha$, one has $2^{n_kL/9}/ 2^{\alpha n_k} \to \infty$, and so the right side is $\geq (\Cr{c: another_new}/2)2^{n_k((L/9)-\beta)}$ for infinitely many $k$. Furthermore, the term $\Cr{C: step_2}n_kLa_{\lfloor \delta n_k \rfloor} \to 0$ as $k \to \infty$, so for any $x >0$, with probability $\geq \Cr{c: another_new}$,
	\begin{equation}\label{eq: new_x_choice}
		N\left( \{t \in [0,1] : \rho_t \leq x\}, 2^{-\frac{n_kL}{9}}\right) \geq \frac{\Cr{c: another_new}}{2} 2^{n_k\left( \frac{L}{9}-\beta\right)} 
	\end{equation}
	for infinitely many $k$. The definition of Minkowski dimension then implies 
	\begin{equation}\label{eq: almost_lower_bound}
		\mathbb{P}\left( \textnormal{dim}^\textnormal{M}(\{t\in [0,1] : \rho_t \leq x\}) \geq 1 - \frac{9 \beta}{L}\right) \geq \Cr{c: another_new}.
	\end{equation}
	Because $\Cr{c: another_new}$ does not depend on $L$ or $x$, we can take $L \to \infty$ for
	\[
	\mathbb{P}\left( \textnormal{dim}^\textnormal{M}(\{t \in [0,1] : \rho_t \leq x\}) = 1\right) \geq \Cr{c: another_new}.
	\]
	Because the sequence $\left(\textnormal{dim}^\textnormal{M}(\{t \in [m,m+1] : \rho_t \leq x\})\right)_{m \geq 0}$ is ergodic (similar to \cite[p.~522]{HPS97} or \cite[Lem.~2.3]{HPS15}), a.s.
	\[
	\text{there exists } m \geq 0 \text{ such that } \textnormal{dim}^\textnormal{M}(\{t \in [m,m+1] : \rho_t \leq x\}) = 1,
	\]
	and by monotonicity of Minkowski dimension, a.s., there exists $s \geq 0$ such that $\textnormal{dim}^\textnormal{M}(\{t \in [0,s] : \rho_t \leq x\}) = 1$. This implies the statement of Theorem~\ref{thm: Minkowski}(2).

	\section{Intermediate cases}\label{sec: formal_statements}
	In this section, we explain how to prove the claims labeled $(\dagger)$, listed above the statement of Theorem~\ref{thm: other_side}. Specifically, we give:
	\begin{Thm}\label{thm: dental_caries}
		Suppose that $F$ satisfies \eqref{eq: critical_def} and \eqref{eq: infinite_sum}.
		\begin{enumerate}
			\item If $\limsup_{k \to \infty} ka_k = \infty$, then for all $x,s \geq 0$,
			\[
			\textnormal{dim}_{\textnormal{M}}\left( \{t \in [0,s] : \rho_t \leq x\}\right) \leq \frac{31}{36} \text{ a.s.}
			\]
			\item If $\liminf_{k \to \infty} ka_k > 0$, then given $\epsilon>0$, there exists $x >0$ such that for all $s\geq 0$,
			\[
			\textnormal{dim}^\textnormal{M}\left( \{t \in [0,s] : \rho_t \leq x\}\right) \leq \frac{31}{36} + \epsilon \text{ a.s.}
			\]
			\item If $\liminf_{k \to \infty} ka_k < \infty$, then for any $\epsilon>0$, there exists $x>0$ such that
			\[
			\lim_{s \to \infty} \mathbb{P}\left( \textnormal{dim}^\textnormal{M}\left( \{ t \in [0,s] : \rho_t \leq x\}\right) > 1-\epsilon\right) = 1.
			\]
		\end{enumerate}
	\end{Thm}
	Assuming the veracity of this theorem for the moment, take $F$ such that $\limsup_{k \to \infty} ka_k = \infty$ but $\liminf_{k \to \infty} ka_k \in (0,\infty)$. Then by item 3, we can choose $x$ and $s$ such that with probability at least $1/2$, the upper dimension of $\{t \in [0,s] : \rho_t \leq x\}$ is greater than $0.99$. For this same $x$ and $s$, item 1 implies that the lower dimension is $\leq 31/36$ a.s. Furthermore, for a different $x>0$ but the same $s$, item 2 implies that the upper dimension is at most $0.95$ a.s. Therefore we conclude the claims in $(\dagger)$.
	
	Because the proof is nearly identical to those appearing earlier in the paper, we only briefly indicate the necessary adjustments. For part 1, only cosmetic changes to the proof of the upper bound of Theorem~\ref{thm: Minkowski}(1) are needed. Instead of applying the argument that follows \eqref{eq: x_choice} to the entire sequence $ka_k$, we apply it to a subsequence $k_\ell a_{k_\ell}$ satisfying $k_\ell a_{k_\ell} \geq \Cr{C: x_inequality} > 0$ for some $\Cr{C: x_inequality}$. In the end, this produces item 1 of Theorem~\ref{thm: dental_caries} instead of the estimate \eqref{eq: end_upper_bound}.
	
	For item 2, we assume that $k a_k \geq \Cr{C: x_inequality}>0$ for some constant $\Cr{C: x_inequality}$. Now we select $x$ such that \eqref{eq: x_choice} holds for this fixed $\Cr{C: x_inequality}$, and then proceed through the ensuing argument. To produce \eqref{eq: almost_end} from the display above it, instead of increasing $\Cr{C: x_inequality}$, we decrease $x$. The rest of the argument is the same, and we end up with an upper bound of $31/36 + \delta$ for the upper Minkowski dimension, as in item 2 of Theorem~\ref{thm: dental_caries} with $\epsilon=\delta$, instead of \eqref{eq: end_upper_bound}. This argument requires $-x \log x$ to be of order $\epsilon$.
	
	In item 3, we assume that $\liminf_{k \to \infty} ka_k < \infty$, and so we can pick a subsequence $(n_k)$ such that $n_k a_{\lfloor \delta n_k \rfloor}$ is bounded. We follow the proof of Theorem~\ref{thm: Minkowski}(2) exactly until step 4 where, in \eqref{eq: new_x_choice}, we choose $x$ larger than $\Cr{C: step_2}n_k L a_{\lfloor \delta n_k \rfloor}$. We obtain \eqref{eq: almost_lower_bound} with this choice of $x$. By ergodicity again, a.s. there exists $s \geq 0$ such that $\textnormal{dim}^\textnormal{M}(\{t \in [0,s] : \rho_t \leq x\}) \geq 1-(9\beta)/L$. This is larger than $1-\epsilon$ assuming we choose $L$ large enough. In the end, our choice of $x$ is of order $1/\epsilon$.

	\bigskip
	\noindent
	\textbf{Acknowledgments.} J. H. thanks Georgia Tech for hospitality during visits related to this work.

\end{document}